\documentclass{amsart}

\usepackage[paperwidth=170mm, paperheight=230mm, total={125mm, 170mm},
 top=29mm,left=23mm,includeheadfoot]{geometry}
\usepackage[utf8]{inputenc}
\usepackage[spanish]{babel}

\usepackage{multicol}

\usepackage{amsmath}
\usepackage{amsthm}
\usepackage{amssymb}
\usepackage{graphicx}
\usepackage{mathrsfs}
\usepackage[colorinlistoftodos]{todonotes}
\usepackage[colorlinks=true, allcolors=black]{hyperref}
\usepackage{comment}
\usetikzlibrary{graphs}
\usepackage{float}
\usepackage{xcolor}
\usepackage{tikz-cd}
\usepackage{manfnt}

\newcommand{\showcomments}{yes}

\newsavebox{\commentbox}
{\ifthenelse{\equal{\showcomments}{yes}}%
{\footnotemark
        \begin{lrbox}{\commentbox}
        \begin{minipage}[t]{1.5in}\raggedright\sffamily\tiny
        \footnotemark[\arabic{footnote}]}
{\begin{lrbox}{\commentbox}}}%
{\ifthenelse{\equal{\showcomments}{yes}}%
{\end{minipage}\end{lrbox}\marginpar{\usebox{\commentbox}}}
{\end{lrbox}}}

\newtheorem{thm}{Theorem}[section]

\newtheorem{theorem}[thm]{Teorema}

\newtheorem{corollary}[thm]{Corolario}
\newtheorem{lemma}[thm]{Lema}
\newtheorem{proposition}[thm]{Proposici\'on}

\theoremstyle{definition}

\newtheorem{definition}[thm]{Definici\'on}

\theoremstyle{remark}

\newtheorem{notation}[thm]{Notaci\'on}

\newtheorem{remark}[thm]{Observaci\'on}
\newtheorem{example}[thm]{Ejemplo}

\newcommand{\dist}{\textup{\textsf{d}}}

\newcommand{\field}[1]{\mathbb{#1}}
\newcommand{\integers}{\ensuremath{\field{Z}}}

\newcommand{\naturals}{\ensuremath{\field{N}}}
\newcommand{\reals}{\ensuremath{\field{R}}}

\title{Grupos y geometr\'ia cubular}
\author{Macarena Arenas}
\address{DPMMS, Centre for Mathematical Sciences, Wilberforce Road, Cambridge, CB3 0WB, UK
 and Clare College, University of Cambridge, Cambridge, CB2 1TL, UK}
	\email{mcr59@cam.ac.uk}
\date{\today}

\makeindex

\begin{document}

\begin{abstract}
Este art\'iculo expositorio es una versi\'on extendida del material presentado en un minicurso de la \emph{Escuela de teor\'ia geom\'etrica de grupos} que se llev\'o a cabo  en el Centro de Ciencias Matem\'aticas en Morelia, en julio de 2022. El objetivo principal es dar una introducci\'on al estudio de los \emph{complejos cubulares no positivamente curvados}, y a trav\'es de estos, una introducci\'on a la teor\'ia geom\'etrica de grupos.
\end{abstract}

\maketitle

\tableofcontents

\section{¿De qu\'e trata este texto?}

Los complejos cubulares no positivamente curvados fueron introducidos originalmente por Gromov~\cite{Gromov87}, quien los exhibi\'o como ejemplos de grupos CAT(0) y de grupos hiperb\'olicos que se pueden entender de manera b\'asicamente combinatoria.
En las \'ultimas d\'ecadas, el estudio de estos objetos ha cobrado importancia en la teor\'ia geom\'etrica de grupos y en la topolog\'ia de dimensiones bajas, alcanzando un m\'aximo (local) con las demostraciones de  dos de los problemas abiertos m\'as importantes en el estudio de 3-variedades:  la \emph{conjetura virtual de Haken} y la \emph{conjetura de fibraci\'on virtual}, demostradas por Agol en~\cite{AgolGrovesManning2012, Agol08} partiendo del trabajo de Wise~\cite{WiseIsraelHierarchy} y sus colaboradores~\cite{BergeronWiseBoundary, HaglundWiseAmalgams, HsuWiseCubulatingMalnormal}.

 En este texto se ha buscado cubrir los temas necesarios para que el lector pueda absorber una narrativa coherente sobre el desarrollo y las aplicaciones de esta teor\'ia, pero no se ha buscado ser enciclop\'edicos: tratamos de evitar las decenas de tangentes posibles, los resultados m\'as especializados, y los temas m\'as `contempor\'aneos'; cuando las demostraciones de los resultados enunciados en el texto son demasiado t\'ecnicas o requieren salirse un poco del tema, las omitimos. Inclu\'imos muchos ejemplos y sugerencias de ejercicios, de manera que el texto podr\'ia servir como material para un curso corto en el tema.

Un primer curso en teor\'ia de grupos y un  curso en topolog\'ia de conjuntos son requisitos absolutamente indispensables para entender este texto, y aunque un lector excepcionalmente motivado podr\'a sacarle provecho a una buena parte del material sin otros conocimientos m\'as avanzados, es esencial para la comprensi\'on de la totalidad del material tener tambi\'en conocimientos b\'asicos de topolog\'ia algebraica (grupo fundamental,  espacios cubrientes, topolog\'ia de superficies). Otros temas deseables, pero no necesarios son un conocimiento b\'asico de geometr\'ia hiperb\'olica, el concepto de CW-complejo, y, si no un conocimiento formal de la teor\'ia, al menos una cultura general respecto al estudio de las 3-variedades.

Todo el material cubierto en este art\'iculo se puede encontrar en otras fuentes, aunque quiz\'a no presentado de esta manera y ciertamente no en espa\~nol. Algunas sugerencias de lectura incluyen~\cite{Sageev95, ChatterjiNiblo04, GGTbook14, WiseCBMS2012}.

\section{Nociones b\'asicas}

\begin{definition}
Un \emph{n-cubo} es una copia de $[0,1]^n$, por ejemplo, un $0$-cubo es un v\'ertice, un $1$-cubo es una arista (un intervalo), un $2$-cubo es un cuadrado y un $3$-cubo es lo que   simplemente llamariamos un cubo. 

\begin{figure}[h!]
\centerline{\includegraphics[scale=0.6]{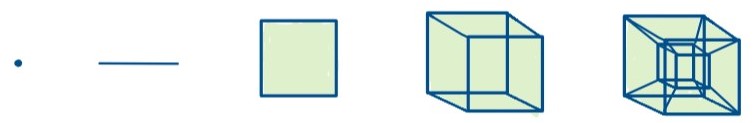}}
\caption{De izquierda a derecha: 0-cubo, 1-cubo, 2-cubo, 3-cubo, 4-cubo.}
\label{fig:0to4}
\end{figure} 

Si $c$ es un $n$-cubo, las \emph{caras} de $c$ son los subcubos que se obtienen al restringir una o varias coordenadas de $c$ a $\{0\}$ o a $\{1\}$.

\begin{figure}[h!]
\centerline{\includegraphics[scale=0.27]{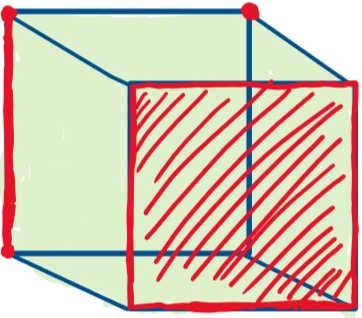}}
\caption{0-,1-, y 2-caras en un 3-cubo.}
\label{fig:caras}
\end{figure} 

Un \emph{complejo cubular} es un complejo celular cuyas celdas son cubos de varias dimensiones pegados a lo largo de sus caras por isometr\'ias. Como espacios topol\'ogicos, los complejos cubulares est\'an equipados con la topolog\'ia cociente. M\'as precisamente, un complejo cubular es el espacio cociente $\mathcal{C}/\mathcal{I}$ donde $\mathcal{C}$ es una uni\'on disjunta de cubos e $\mathcal{I}$ es una colecci\'on de isometr\'ias entre caras de cubos en $\mathcal{C}$ (aqu\'i asumimos adem\'as que las caras solo se pegan a caras de la misma dimensi\'on). 

Escribir las isometr\'ias de manera expl\'icita en general ser\'ia un engorro, as\'i que, en los ejemplos, adoptaremos la convenci\'on de describir las identificaciones entre las distintas caras etiquet\'andolas con letras, colores o flechas.

La \emph{dimensi\'on} de un complejo cubular es  el supremo de las dimensiones de sus cubos; decimos que un complejo cubular es \emph{finito} si tiene un n\'umero finito de cubos, y que es \emph{localmente finito} si ninguno de sus puntos est\'a contenido en una cantidad infinita de cubos.
\end{definition}

\begin{figure}[h!]
\centerline{\includegraphics[scale=0.6]{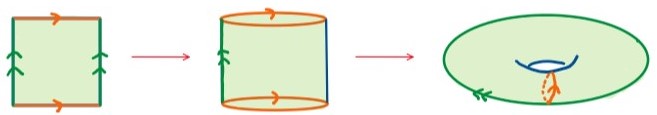}}
\caption{Un toro es un complejo cubular.}
\label{fig:toro}
\end{figure} 

\begin{definition} La \emph{aureola} de un v\'ertice $v$ en un complejo cubular $X$, denotada $link(v)$ es el complejo celular cuyas celdas son los $n$-simplejos correspondientes a las esquinas de los $(n+1)$-cubos adjacentes a $v$.
\end{definition}

Intuitivamente, $link(v)$ es la $\epsilon$-esfera alrededor de $v$, donde $\epsilon>0$ es suficientemente peque\~no. Esto se puede formalizar como sigue. Si $X$ es un complejo cubular, cada $n$-cubo $c$ de $X$ es isom\'etrico a $[0,1]^n$, y por lo tanto tiene sentido hablar de la esfera de radio $\epsilon$ (donde, digamos, $\epsilon < \frac{1}{2}$) con centro un v\'ertice $v$: esta es la aureola de $v$ en $c$. La aureola de $v$ en $X$ se puede obtener ahora tomando la aureola de $v$ en cada cubo de $X$ que contiene a $v$, y pegando estas aureolas a lo largo de sus caras de acuerdo  a las identificaciones inducidas por los cubos correspondientes en $X$.

\begin{definition}
Una \emph{gr\'afica simplicial} es una gr\'afica que no tiene b\'igonos (ciclos de longitud 2) ni lazos (ciclos de longitud 1).
Un complejo  $Y$ es \emph{bandera} si es un complejo simplicial y $n+1$ v\'ertices bordean un $n$-simplejo en $Y$ si y solo si son adjacentes. 
\end{definition}

La idea, intuitivamente, es que un complejo $Y$ es bandera si siempre que vemos el $n$-esqueleto de un simplejo en $Y$, este est\'a relleno con un $n$-simplejo.
De esta manera los complejos bandera se encuentran en correspondencia biyectiva con las gr\'aficas simpliciales.

\begin{figure}[h!]
\centerline{\includegraphics[scale=0.38]{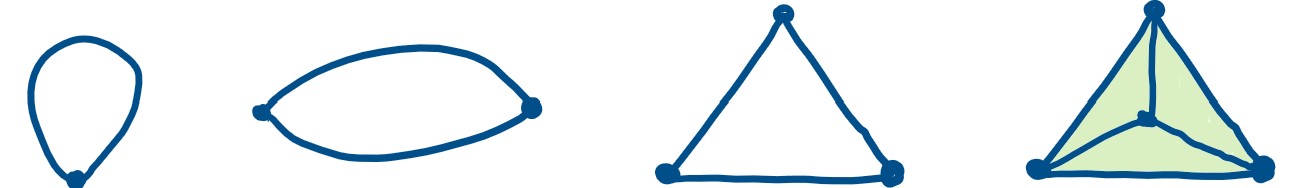}}
\caption{Algunos complejos que no son bandera: los primeros dos ni siquiera son simpliciales, a los otros dos les falta el simplejo de dimensi\'on m\'axima determinado por sus esqueletos (es dif\'icil dibujar en tres dimensiones, pero imag\'inense que el tetrahedro no est\'a relleno).}
\label{fig:notflag}
\end{figure} 

\begin{definition}
Un complejo cubular $X$ es \emph{no positivamente curvado} (`npc') si la aureola de cada v\'ertice en $X$ es un complejo bandera. 
Un complejo cubular es \emph{CAT(0)} si es npc y simplemente conexo. Esto es, los complejos cubulares CAT(0) corresponden exactamente a los cubrientes universales de complejos cubulares npc. 
\end{definition}

\begin{figure}[h!]
\centerline{\includegraphics[scale=.3]{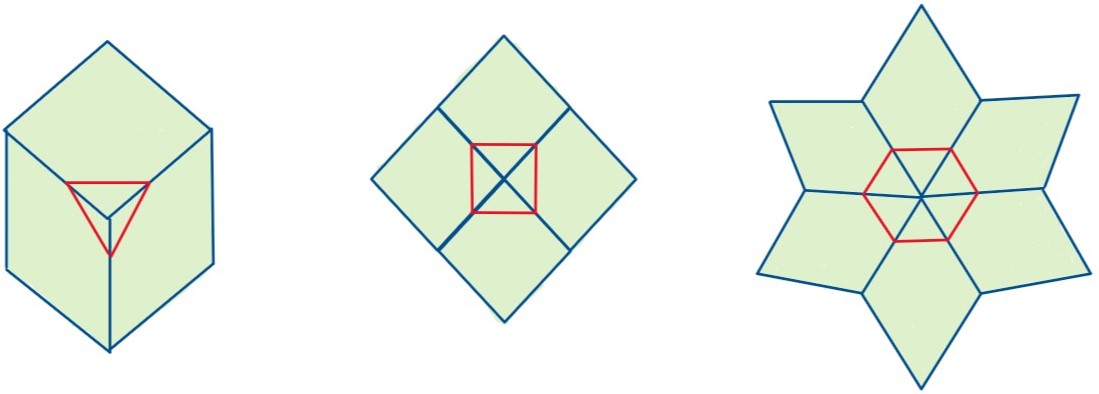}}
\caption{Ejemplos de aureolas de v\'ertices en complejos cubulares. La primera aureola solo puede ser un complejo bandera si hay un 3-cubo pegado en ese v\'ertice; las otras dos s\'i son bandera.}
\label{fig:exlk}
\end{figure} 

Los complejos cubulares CAT(0) de dimensi\'on finita son espacios CAT(0) en el sentido m\'etrico de Cartan-Alexandrov-Toponogov~\cite{BridsonHaefliger}, pero no entraremos en detalle en este aspecto de la teor\'ia, y en general no necesitaremos usar la m\'etrica CAT(0). Para nosotros, la consecuencia m\'as importante de la existencia de una m\'etrica CAT(0) es que los complejos cubulares CAT(0) son contra\'ibles, y por lo tanto los complejos cubulares npc son espacios clasificantes para sus grupos fundamentales. Una demostraci\'on que no utiliza la m\'etrica CAT(0) se puede encontrar en~\cite{arenas2023asphericity}.

\subsection{Un zool\'ogico de ejemplos}

\begin{example} Toda gr\'afica es un complejo cubular npc, puesto que $link(v)$ no tiene aristas y por lo tanto satisface la condici\'on de ser bandera autom\'aticamente.
\end{example}

\begin{example} El espacio euclideano $\mathbb{E}^n$, teselado por $n$-cubos de la manera usual, es un complejo cubular CAT(0). En efecto, la aureola de todo v\'ertice es una $(n-1)$-esfera teselada por  $2^{n}$ simplejos, y por lo tanto es un complejo bandera.
\end{example}

\begin{example} Al igual que el toro de dimensi\'on $2$ en la figura~\ref{fig:0to4}, para cada $n$, el $n$-toro $T^n$ es un complejo cubular no positivamente curvado. Lo podemos ver como el espacio cociente obtenido al tomar un $n$-cubo e identificar pares de caras opuestas como en la figura~\ref{fig:3toro}.
\end{example}

\begin{figure}[h!]
\centerline{\includegraphics[scale=.3]{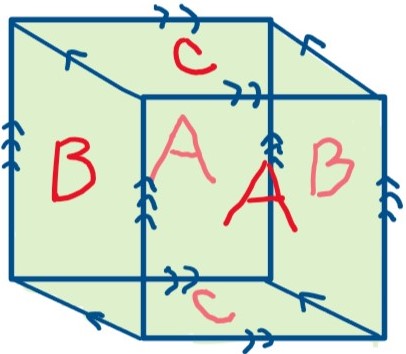}}
\caption{Como obtener un 3-toro identificando caras de un 3-cubo.}
\label{fig:3toro}
\end{figure} 

\begin{remark} \label{rem:girth}
El \emph{cuello} de una gr\'afica $\Gamma$  es la longitud del ciclo más corto contenido en dicha gr\'afica. Si $\Gamma$ no tiene ciclos, entonces definimos el cuello de $\Gamma$ como $\infty$.
Un complejo cubular $X$ de dimensi\'on 2 es npc si y solo si para cada v\'ertice $v$ en $X$ el cuello de $link(v)$ es $\geq 4$.
\end{remark}

\begin{theorem} \label{thm:surface-cube}
Toda superficie cerrada con caracter\'istica de Euler $\leq 0$ es homeomorfa a un complejo cubular npc.
\end{theorem}

\begin{proof} Demostraremos el resultado en el caso en el que la superficie es orientable; el caso no-orientable es an\'alogo. Sea $P$ un $4n$-\'agono regular con lados identificados para obtener una superficie de g\'enero $n$. Subdividimos $P$ a\~nadiendo un v\'ertice $x$ en el centro de $P$, un v\'ertice $y_i$ en el centro de cada arista de $P$ y una arista conectando $y_i$ a $x$ para cada $i \in \{1, \ldots, 4n\}$. El complejo resultante $S$ es claramente un complejo cubular, hay que verificar que es npc examinando las aureolas de los v\'ertices. Para cada $i$, la aureola $link(y_i)$ es un ciclo de longitud $4$, y tanto $link(x)$ como  $link(v)$ son ciclos de longitud $4n \geq 4$, as\'i que por la observaci\'on~\ref{rem:girth}, $S$ es npc. 
\end{proof}

\begin{figure}[h!]
\centerline{\includegraphics[scale=.3]{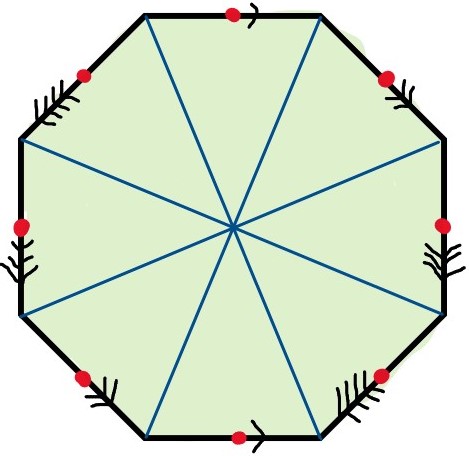}}
\caption{Ejemplo de la subdivisi\'on de $S$ en la demostraci\'on del teorema~\ref{thm:surface-cube} para $g=2$.}
\label{fig:cubS2}
\end{figure}

\subsubsection{¿Qu\'e operaciones topol\'ogicas preservan la propiedad de ser npc?} 

\hfill\\

\textbf{Subdivisi\'on:} La subdivisi\'on baric\'entrica de un complejo cubular npc es de nuevo un complejo cubular npc (ejercicio).

\textbf{Productos:} El producto de complejos cubulares npc tambi\'en es un complejo cubular npc (ejercicio). Esto nos da una manera alternativa de verificar que $\mathbb{E}^n$ y $T^n$ son npc usando solamente que $\mathbb{E}^1$ y $S^1$, por ser gr\'aficas, son npc.

\textbf{Cubrientes:} Si $X$ es un complejo cubular npc y $\widehat{X} \xrightarrow{p} X$ es un mapeo cubriente, entonces $\widehat{X}$ es un complejo cubular npc. Las celdas de $X$ se levantan a celdas de $\widehat{X}$. Como $p$ es un mapeo cubriente, las aureolas de los levantamientos de los v\'ertices en $\widehat{X}$ son id\'enticas a las aureolas de los v\'ertices correspondientes en $X$. Esto nos da una manera alternativa de verificar, por ejemplo, que las superficies cerradas con caracter\'istica de Euler $< 0$ son homeomorfas a complejos cubulares npc: basta con dar esta estructura a una superficie que est\'e cubierta por todas las dem\'as (ejercicio: ¿cu\'al es esta superficie?)

\textbf{Subcomplejos:}\label{subcomp} Si $X$ es un complejo cubular npc de dimensi\'on 2, entonces todo subcomplejo de $X$ es tambi\'en npc. De esto se sigue, como un corolario del teorema~\ref{thm:surface-cube}, que las superficies con frontera son homeomorfas a complejos cubulares npc.  Sin embargo, ser npc no se preserva al tomar subcomplejos en general: esto es claro al considerar por ejemplo el $(n-1)$-esqueleto de un $n$-cubo para $n \geq 3$.
Veremos mas adelante  que los subcomplejos ``convexos'' de complejos cubulares npc tambi\'en son npc. 

\subsubsection{Grupos de Artin de \'angulo recto}

\begin{definition} 
Sea $\Gamma$ una gr\'afica simplicial con v\'ertices $V$ y aristas $E$,  el \emph{grupo de Artin de \'angulo recto} (`gaar', para abreviar) $G(\Gamma)$ es el grupo dado por la presentaci\'on:
\begin{equation}\label{refeq:raag}
G(\Gamma)= \langle x_v : v \in V | [x_u,x_v] : \{u,v\} \in E \rangle
\end{equation}
donde $[x_u,x_v]=x^{-1}_ux^{-1}_vx_ux_v$ es el conmutador de $x_u$ y $x_v$.

\begin{remark} Es un teorema de Droms \cite{DromsPhD83} que los gaars est\'an completamente determinados por sus gr\'aficas:  $G(\Gamma)$ es isomorfo a $G(\Gamma')$ si y solo si $\Gamma$ es isomorfa a $\Gamma'$.
\end{remark}

\begin{definition}
El \emph{complejo de presentaci\'on} $X(\mathcal{P})$ asociado a una presentación $$\mathcal{P}=\langle a_1, \ldots, a_s | r_1, \ldots, r_k\rangle$$ es el complejo de celdas que se construye como sigue: empezamos con un bouquet de lazos $B$ donde cada lazo $l_s$ corresponde  a un generador $a_i$ en $\mathcal{P}$. A este bouquet le pegamos una 2-celda $D_j$ por cada relación $r_j$ en $\mathcal{P}$, donde la frontera de $D_j$ esta pegada a $B$ a lo largo del camino correspondiente a $r_j$ en $\mathcal{P}$. Por el teorema de Seifert Van-Kampen, el grupo fundamental $\pi_1 X(\mathcal{P})$ es isomorfo al grupo presentado por $\mathcal{P}$.
\end{definition}

\begin{notation}
Dado un complejo cubular $X$, escribimos $X^{(k)}$ para denotar a su $k$-esqueleto. As\'i,  $X^{(0)}$ son los v\'ertices de $X$, $X^{(1)}$ son las aristas, $X^{(2)}$ son los cuadrados, etc\'etera.
\end{notation}

\begin{theorem}\label{thm:salvettis} Todo gaar $G(\Gamma)$ es isomorfo a $\pi_1R$ para alg\'un complejo cubular npc $R$. Si $\Gamma$ es finita, entonces $R$ tambi\'en lo es.
\end{theorem}

\begin{proof}
Definimos $R^{(2)}=X(\mathcal{P})$, donde $X(\mathcal{P})$ es el complejo de presentaci\'on asociado a la presentaci\'on~(\ref{refeq:raag}). Si el cuello de  $\Gamma$ es $\geq 4$, entonces $R^{(2)}= R$ es el complejo cubular deseado, en caso contrario, podemos completar $R^{(2)}$ a un complejo cubular npc $R$ pegando $k$-cubos inductivamente siempre que sus esqueletos est\'en presentes. Esto es, por cada subgr\'afica completa $K_n$ en $\Gamma$, $R$ tiene un $n$-toro cuyo 1-esqueleto est\'a dado por los v\'ertices en $K_n$ y cuyo 2-esqueleto est\'a dado por las aristas en $K_n$. Hay que verificar que $link(w)$ es bandera, donde $w$ es el \'unico v\'ertice de $R$. Es claro de la presentaci\'on dada en~\eqref{refeq:raag} que  $link(w)$ es simplicial; si la frontera de un $n$-simplejo aparece en $link(w)$, esto quiere decir que hay $n$ generadores en~\eqref{refeq:raag} que conmutan, o sea que hay un $K_n$ correspondiente a estos generadores en $\Gamma$, por lo que el resultado se sigue de que $R$ tiene un $n$-cubo correspondiente a cada $K_n$ en $link(w)$.
\end{proof}

El complejo $R$ en el teorema~\ref{thm:salvettis} se conoce como el \emph{complejo de Salvetti} asociado a $G(\Gamma)$.

\begin{figure}[h!]
\centerline{\includegraphics[scale=0.5]{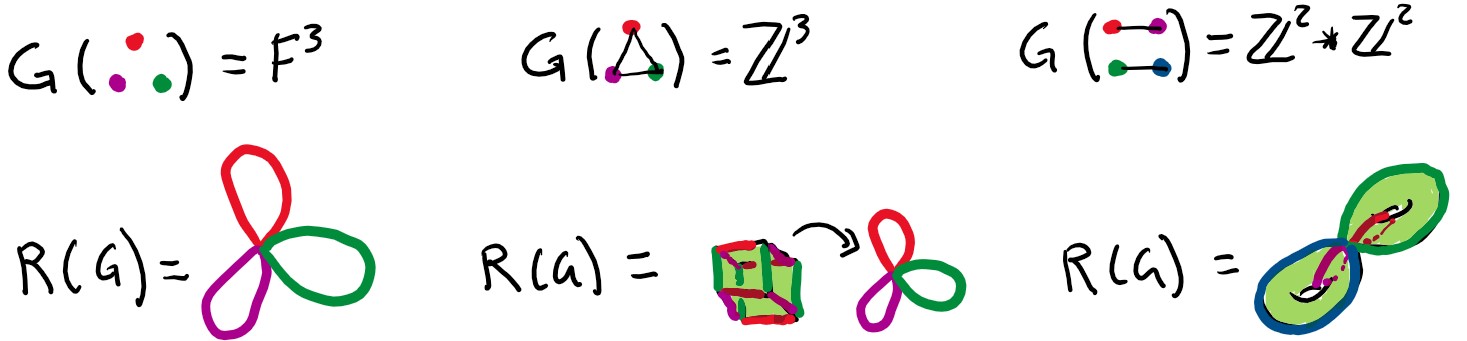}}
\caption{Algunos ejemplos de gaars y sus complejos de Salvetti.}
\label{fig:exraags}
\end{figure}
\end{definition}

\subsection{Sobre la m\'etrica en complejos cubulares}

Como mencionamos antes, los complejos cubulares CAT(0) de dimensi\'on finita admiten una m\'etrica CAT(0), pero tambi\'en admiten otras m\'etricas. Lo m\'as \'util para nosotros, en general, ser\'a dar al 1-esqueleto de  $\widetilde X$ la m\'etrica usual para gr\'aficas. Esto es, la distancia $\dist(v,v')$ entre v\'ertices $v,v'$ es el \'infimo n\'umero de aristas en caminos que conectan a $v$ y $v'$.
Est\'a m\'etrica se extiende a la m\'etrica $\ell^1$ en $\widetilde X$, de tal manera que el 1-esqueleto de  $\widetilde X$ se encaja de manera isom\'etrica en  $\widetilde X$. Estas tres m\'etricas -- la m\'etrica CAT(0) en $\widetilde X$, la m\'etrica para gr\'aficas en el 1-esqueleto de  $\widetilde X$, y su extensi\'on a  $\widetilde X$ -- son cuasi-isom\'etricas en el sentido de la definici\'on~\ref{def:qi}.

\section{Hiperplanos e isometr\'ias locales}

\subsection{Hiperplanos en general}

\begin{definition}
En un complejo cubular $X$ un \emph{cubo central} es la restricci\'on de exactamente una coordenada de un $n$-cubo $c=[0,1]^n$ a $\{\frac{1}{2}\}$. Otra manera de decir esto es que un cubo central es un $(n-1)$-cubo que pasa por el baricentro de $c$ y que es paralelo a alg\'un par de caras (necesariamente opuestas) de $c$. Aqu\'i es importante notar que un cubo central es un subespacio de $c$, pero no un subcomplejo. Cada $n$-cubo $c$ tiene $n$ cubos centrales de dimensi\'on $n-1$ distintos, todos los cuales se intersecan en el baricentro de $c$.
\end{definition}

\begin{figure}[h!]
\centerline{\includegraphics[scale=0.37]{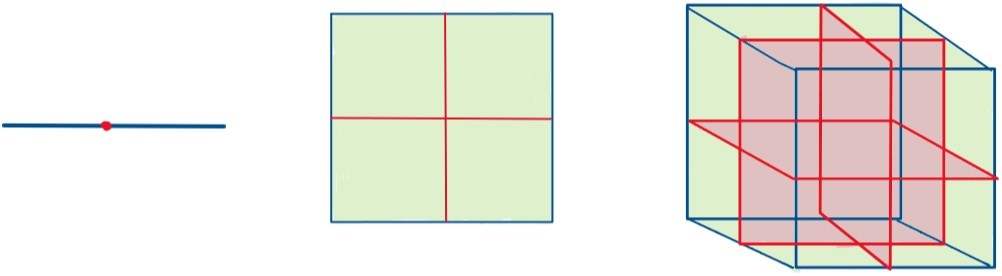}}
\caption{Cubos centrales en cubos de varias dimensiones.}
\label{fig:midcubes}
\end{figure}

\begin{definition}
Si $e$ y $e'$ son cubos centrales en cubos distintos de $X$, entonces o su intersecci\'on es vac\'ia o es un cubo central de una cara de $X$.
Sea $S$ la union ajena de todos los cubos centrales en $X$. Sea $\bar S$ el cociente de $S$ que se obtiene al identificar los cubos centrales a lo largo de las caras en que se intersecan. Un \emph{hiperplano} $H$ es un componente conexo de $\bar S$. N\'otese que la inclusi\'on $S \hookrightarrow X$ induce un mapeo $\bar S \rightarrow X$, y por lo tanto un mapeo $H \rightarrow X$.
\end{definition}

Un hiperplano es \emph{dual} a todo 1-cubo que lo interseca.

\begin{example} Si $X$ es una gr\'afica, entonces sus hiperplanos son los baricentros de las aristas. 
\end{example}

\begin{example} Si $X$ es el espacio euclideano $\mathbb{E}^n$, cubulado de la manera usual, sus hiperplanos son copias de $\mathbb{E}^{n-1}$.
\end{example}

\begin{example} Si $X$ es la cubulaci\'on de la superficie $S$ de g\'enero $g$ detallada en la demostraci\'on del teorema~\ref{thm:surface-cube}, entonces los hiperplanos de $X$ corresponden a curvas cerradas en $S$. 
\end{example}

\begin{figure}[h!]
\centerline{\includegraphics[scale=0.35]{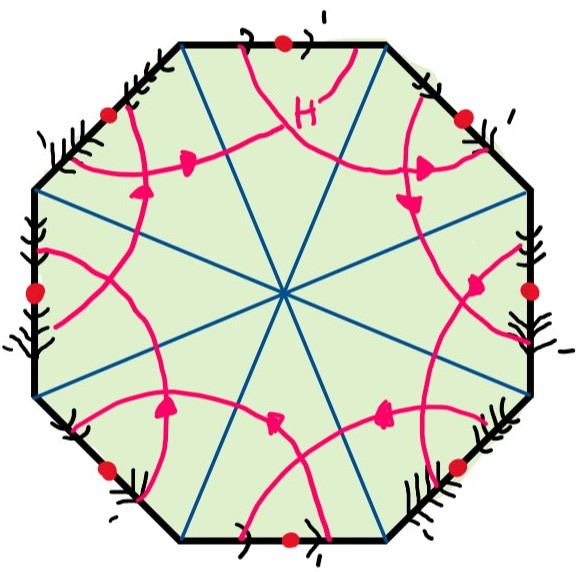}}
\caption{Un posible hiperplano en una superficie de g\'enero 2.}
\label{fig:hipsup}
\end{figure}

\begin{example} Si $X$ es un $n$-toro, cubulado de la manera usual, entonces $X$ tiene $n$ hiperplanos, y estos son todos $(n-1)$-toros, cada par de estos se interseca en un $(n-2)$-toro, cada terna en un $(n-3)$-toro, etc\'etera.
\end{example}

\begin{figure}[h!]
\centerline{\includegraphics[scale=0.7]{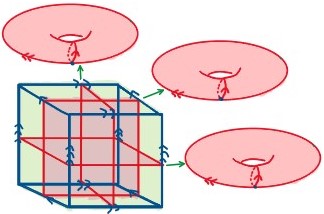}}
\caption{Hiperplanos en un $3$-toro.}
\label{fig:subtoros}
\end{figure}

A grosso modo, los hiperplanos son a complejos cubulares npc lo que las geod\'esicas son a superficies y las superficies m\'inimas son a las  $3$-variedades. De hecho uno puede, en muchos casos, cubular estos objetos de tal suerte que  ciertas colecciones de curvas o superficies  correspondan exactamente a los hiperplanos en la cubulaci\'on. 

\begin{definition}
El \emph{soporte} $N(H)$ de un hiperplano $H$ de $X$ es la vecindad cubular de $H$. Es decir, es el complejo cubular que se obtiene de ``engordar'' cada cubo central de $H$ al cubo que lo contiene en $X$. 
\end{definition}

\begin{figure}[h!]
\centerline{\includegraphics[scale=0.5]{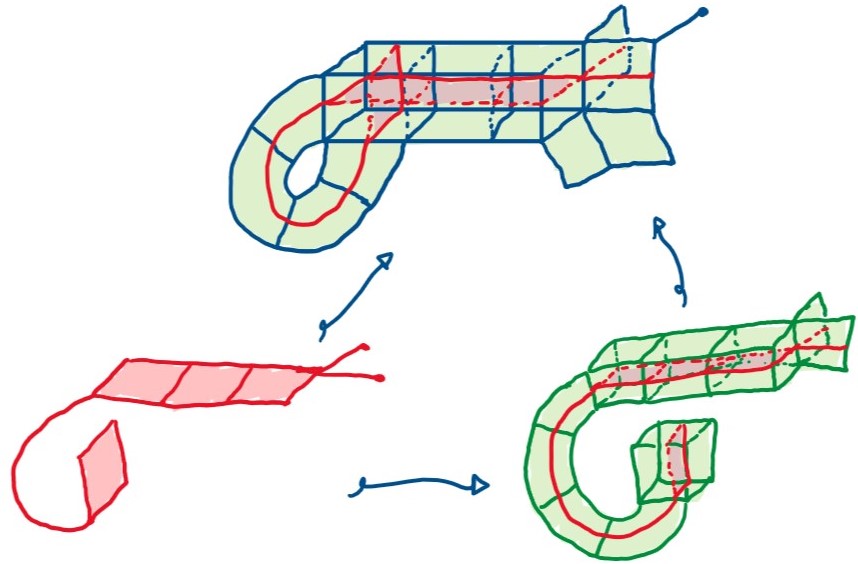}}
\caption{Un hiperplano y su soporte.}
\label{fig:hipsop}
\end{figure} 

Un mapeo celular es \emph{combinatorio} si respeta la dimensi\'on de las celdas. Esto es, $\varphi:X \rightarrow Y$  es un mapeo combinatorio si es un homeom\'orfismo al restringirlo al interior de las celdas de $X$.

\begin{definition}
Un mapeo combinatorio $\varphi:X \rightarrow Y$ entre complejos cubulares npc es una \emph{isometr\'ia local} si es localmente inyectivo y  si para cada par de  1-cubos $e, f$ adyacentes a un 0-cubo $y$ se cumple lo siguiente: si $e,f$  se mapean a 1-cubos $\varphi(e),\varphi(f)$ en $Y$, y $\varphi(e),\varphi(f)$ bordean la esquina de un cuadrado alrededor de $\varphi(y)$, entonces $e, f$ bordean la esquina de un cuadrado alrededor de $y$.
\end{definition}

\begin{figure}[h!]
\centerline{\includegraphics[scale=0.75]{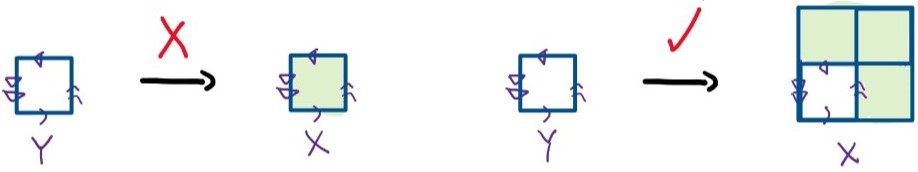}}
\caption{En la izquierda un mapeo que no es una isometr\'ia local; en la derecha una isometr\'ia local. Aqu\'i las flechas indican a d\'onde se est\'a mapeando cada arista de $Y$ en $X$.}
\label{fig:localex}
\end{figure} 

\begin{example} Toda inmersi\'on de gr\'aficas es una isometr\'ia local.
\end{example}

\begin{example} Todo mapeo cubriente $\hat X  \rightarrow X$ de complejos cubulares npc es una isometr\'ia local.
\end{example}

\begin{lemma}\label{lem:isosop} Para todo hiperplano $H$ de $X$, el mapeo $N(H) \looparrowright X$ es una isometr\'ia local.
\end{lemma}

\begin{proof}
Sea $\bar c$ un cuadrado en la imagen, con 1-cubos  $\bar e, \bar e'$ y $\bar f, \bar f'$, de tal forma que $\bar e$ y $\bar f$ comparten un v\'ertice $\bar v$. 
Tenemos dos casos: primero, si $H$ es dual a  $\bar e$ o a $\bar f$ en $N(H)$, entonces necesariamente se extiende a $\bar c$, y es dual a la cara opuesta $\bar e'$ o $\bar f'$, por lo tanto hay un cuadrado correspondiente $c$ en $N(H)$.
Si $H$ no es dual a $\bar e$ ni a $\bar f$, entonces, como $\bar c$ est\'a en la imagen de $N(H) \looparrowright X$,  $H$ es dual a un 1-cubo $\bar g$ que comparte el v\'ertice $\bar v$ con $\bar e$ y $\bar f$. As\'i, $ e, g$ y $f, g$ forman esquinas de cuadrados  $ c'$ y $ c''$ en $N(H)$. Los cuadrados correspondientes en la imagen forman, junto con $\bar c$, la esquina de un 3-cubo, y como $X$ es npc, esto implica que existe un 3-cubo $\bar K$ rellenando esa esquina. Ahora, $H$ se extiende al interior de $\bar K$, y por lo tanto existe un 3-cubo $K$ en $N(H)$ que se mapea a $\bar K$, as\'i que $e, f$ son la esquina de un cuadrado -- una cara de $K$ -- en $N(H)$.
\end{proof}

\subsection{Hiperplanos en complejos CAT(0) y convexidad}

En el caso de los complejos cubulares CAT(0), podemos parafrasear la definici\'on de hiperplano de manera un poco m\'as elegante: 

\begin{definition}\label{def:hip2}
Un \emph{hiperplano} en un complejo cubular CAT(0) $\widetilde X$ es un subespacio conexo que interseca a cada cubo de $\widetilde X$ en un  \'unico cubo central o en el vac\'io.
\end{definition}

Podemos definir los soportes de los hiperplanos de manera an\'aloga:

\begin{definition}
El \emph{soporte} $N(H)$ de un hiperplano $H$ de $\widetilde X$ es la uni\'on de todos los cubos que contienen a $H$.
\end{definition}

¿Por que no dar esta definici\'on desde el principio? Bueno, una raz\'on es que al definir los hiperplanos como en~\ref{def:hip2}, no est\'a claro en principio que estos existan. Eso es parte de lo que implica el siguiente:

\begin{theorem}\label{thm:cat0}
Sea $\widetilde X$ un complejo cubular CAT(0), entonces (\cite{Sageev95}):

\begin{enumerate}

\item \label{it:cat1} Cada cubo central yace en un \'unico hiperplano de $\widetilde X$.
\item \label{it:cat2} Cada hiperplano $H$ de $\widetilde X$ es un complejo cubular CAT(0).
\item \label{it:cat3} $N(H)\cong H\times[0, 1]$  y es un subcomplejo convexo de $\widetilde X$.
\item \label{it:cat4} $\widetilde X - H$ tiene exactamente dos componentes conexas.
\item \label{it:cat5} Si $H$ y $H'$ son hiperplanos que se cruzan en $\widetilde X$, y  si $e$ y $e'$ son 1-cubos duales a $H$ y $H'$ con un v\'ertice en com\'un, entonces hay un cuadrado en $\widetilde X$ que contiene a $e$ y $e'$.  
\end{enumerate}
\end{theorem}

Recordemos que el $1$-esqueleto $\widetilde{X^{(1)}}$ de un complejo cubular CAT(0) $\widetilde X$ se puede ver como un espacio m\'etrico donde la m\'etrica es la m\'etrica usual para gr\'aficas. As\'i, la distancia entre v\'ertices $v,v'$ es igual al n\'umero de aristas en un camino de longitud m\'inima entre $v$ y $v'$. Una \emph{geod\'esica} es un camino de longitud m\'inima en esta m\'etrica.

\begin{lemma}\label{lem:geodesic} Un camino $\sigma \rightarrow \widetilde X$ es una geod\'esica si y solo si cada arista de $\sigma$ es dual a un hiperplano distinto de $\widetilde X$.
\end{lemma}

\begin{proof}[Sketch]
Procedemos por contradicci\'on: supongamos que  $\sigma$ es una geod\'esica y que existe un hiperplano $H$ que cruza a $\sigma$ m\'as de una vez. Sea $h$ un camino de longitud m\'inima en $H$ que interseca a un subcamino $\tau$ de $\sigma$ en dos puntos, y supongamos adem\'as que $H$ y $\tau$ est\'a elegidos para satisfacer lo siguiente: entre todos los hiperplanos que cruzan a $\sigma$ m\'as de una vez, $H$ tiene la propiedad de que el subcamino  $\tau \subset \sigma$ tiene longitud m\'inima entre todas las posibles elecciones de subcaminos de $\sigma$ cruzados m\'as de una vez por alg\'un hiperplano. 

La concatenaci\'on $h\tau^{-1}$ es la frontera de un disco $D$ en $\widetilde X$ (como $h\tau^{-1}$ es un camino cerrado, necesariamente es nullhomot\'opico en $\widetilde X$). Podemos asumir adem\'as que $D$ yace en el $2$-esqueleto de $\widetilde X$. Consid\'erese el camino $\sigma'$ que se obtiene al substituir $\tau$ con $h$ como en la figura~\ref{fig:geodesicsketch}. Por las elecciones que hicimos en el p\'arrafo anterior,  cada hiperplano dual a $h$ en $D$ cruza a $\tau$, podemos concluir que $h$ es m\'as corto que $\tau$, y que por lo tanto $\sigma'$ es m\'as corto que $\sigma$, contradiciendo la suposici\'on de que $\sigma$ es una geod\'esica.
\end{proof}

\begin{figure}[h!]
\centerline{\includegraphics[scale=0.5]{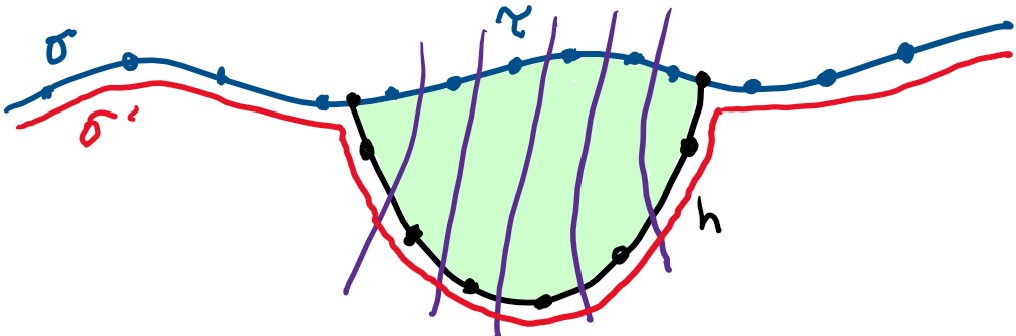}}
\caption{Idea de la demostraci\'on del lema~\ref{lem:geodesic}.}
\label{fig:geodesicsketch}
\end{figure} 

Un subcomplejo $S \subset X$  es \emph{pleno} si siempre que un subcomplejo en $S$  bordea un $n$-cubo  en $X$, entonces este cubo tambi\'en yace en $S$.

\begin{definition} Un subcomplejo pleno $S \subset X$ es \emph{convexo} si para toda geod\'esica $\sigma \rightarrow X$ cuyos extremos yacen en $S$, tenemos que $\sigma \subset S$.  
\end{definition}

\begin{figure}[h!]
\centerline{\includegraphics[scale=0.35]{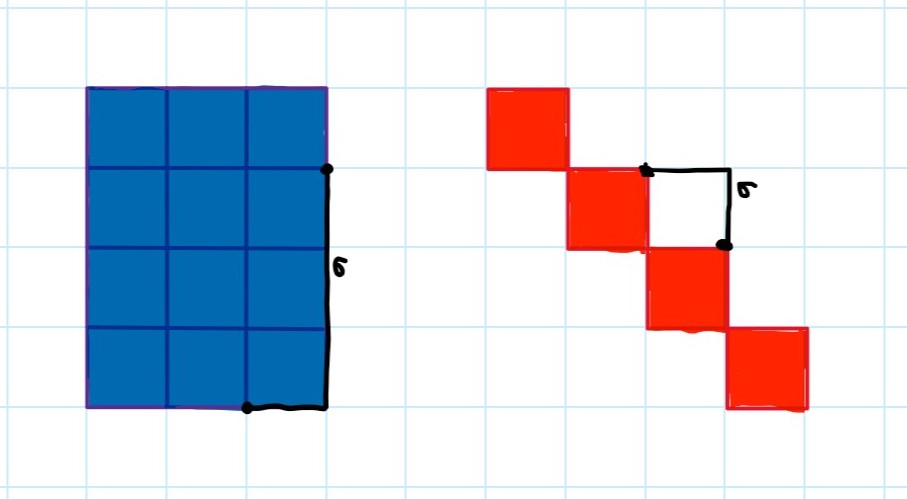}}
\caption{Un subcomplejo convexo (izquierda) y un subcomplejo no convexo (derecha) de la cubulaci\'on est\'andar de $\mathbb{E}^2$.}
\label{fig:convexex}
\end{figure} 

\begin{remark}\label{rmk:combconv}
Sea $Y$  un subcomplejo conexo  de un complejo cubular CAT(0) $\widetilde X$. Si para cada $n$-cubo $c$  con $n \geq 2$ en  $\widetilde X$ se satisface que si toda una esquina de $c$ yace en $Y$, entonces $c$ yace en $Y$, entonces $Y$ es convexo.
\end{remark}

\begin{theorem} \label{thm:local} Si $\varphi: Y \rightarrow X$  es una isometr\'ia local, entonces:
 \begin{enumerate}
 
 \item \label{it:isom2} El mapeo inducido $\tilde{\varphi}: \widetilde{Y} \rightarrow \widetilde{X}$ es inyectivo, 
 \item \label{it:isom3} $\widetilde{Y}\hookrightarrow \widetilde{X}$ es un subcomplejo convexo, y
 \item \label{it:isom1} $\varphi_*: \pi_1 Y \rightarrow \pi_1 X$ es inyectiva.
 \end{enumerate}

\end{theorem}

\begin{proof}

Demostraremos solamente la \'ultima parte del teorema, pues las primeras dos partes requieren usar propiedades de diagramas de disco que no hemos discutido en este art\'iculo (el lector puede consultar~\cite{WiseCBMS2012} si quiere aprender sobre diagramas de disco). Si $\sigma \rightarrow Y$ no es nullhomot\'opico en $Y$, entonces se levanta a un camino $\tilde \sigma \rightarrow \widetilde Y$ que no se cierra. Como $\varphi$ es una isometr\'ia local,  la parte~\eqref{it:isom2} del teorema implica que $\tilde \varphi(\tilde \sigma) \rightarrow \widetilde X$ no puede ser un camino cerrado, y por lo tanto, $\varphi(\sigma) \rightarrow X$ no es nullhomot\'opico. 
\end{proof}

\begin{corollary} El teorema~\ref{thm:local} tiene varias consecuencias inmediatas:
\begin{enumerate}
\item Si $\widehat X \rightarrow X$ es un mapeo cubriente, entonces $\pi_1 \widehat X < \pi_1 X$.
\item Si $H$ es un hiperplano de $X$, entonces $\pi_1 H$ es isomorfo a un subgrupo de $\pi_1X$.
\item Para toda superficie $S$ con $\chi(S) <0$ existe un subgrupo libre no abeliano $F_2 < \pi_1 S$.
\item Si $\Gamma'$ es una subgr\'afica plena de $\Gamma$, entonces $G(\Gamma') < G(\Gamma)$.
\end{enumerate}
\end{corollary}

\begin{proof}
La primera parte se sigue de que los mapeos cubrientes son isometr\'ias locales (claro que sabemos por topolog\'ia algebraica que este resultado es cierto en general, y que de hecho es una correspondencia biyectiva).
La segunda parte se sigue del lema~\ref{lem:isosop} y del teorema~\ref{thm:local}. 
Para la tercera y la cuarta parte, hay que encontrar isometr\'ias locales $Y \rightarrow X$, donde $Y$ tiene grupo fundamental libre o es el gaar correspondiente a una subgr\'afica plena, y $X$ es un complejo cubular npc con $\pi_1 X \cong \pi_1 S$ o $X$ es el complejo de Salvetti asociado a $G(\Gamma)$. Esto queda como ejercicio.
\end{proof}

\begin{lemma}\label{lem:convexisnpc} Si $X$ es un complejo cubular CAT(0) y $Y \subset X$ es un subcomplejo convexo, entonces $Y$ es CAT(0).
\end{lemma}

\begin{proof}
Como $Y \subset X$ es un subcomplejo convexo, entonces $Y$ es un subcomplejo pleno de $X$ y es npc porque las aureolas de v\'ertices en $Y$ son restricciones de las aureolas correspondientes en $X$ y las aureolas en $Y$  son subcomplejos plenos de las aureolas en $X$; $Y$ es simplemente conexo porque la inclusi\'on $i:Y \rightarrow X$ es una isometr\'ia local por la conclusi\'on~\eqref{it:isom1} del teorema~\ref{thm:local}.
\end{proof}

\section{Cuasiconvexidad y grupos hiperb\'olicos}

En esta secci\'on nos desviaremos ligeramente del contexto cubular. Le pedimos al lector tener un poco de paciencia: la importancia de este desv\'io quedar\'a clara en las secciones posteriores.

\begin{definition} Sea $G$ un grupo dado por una presentaci\'on $\langle S |R \rangle$. La \emph{gr\'afica de Cayley de $G$ con respecto a $S$} es la gr\'afica $Cay(G,S)$ cuyos v\'ertices son los elementos de $G$, y donde dos v\'ertices $g$ y $g'$ est\'an conectadas por una arista dirigida de $g$ a $g'$ si $g(g')^{-1} \in S$.
\end{definition}

Sea $G$ un grupo finitamente generado y sea $Cay(G,S)$ su gr\'afica de Cayley con respecto a un conjunto finito de generadores $S$;  $Cay(G,S)$  es un espacio m\'etrico con la m\'etrica usual de las gr\'aficas. Ver las figuras~\ref{fig:enteros} y~\ref{fig:free} para algunos ejemplos.

\begin{figure}[h!]
\centerline{\includegraphics[scale=0.45]{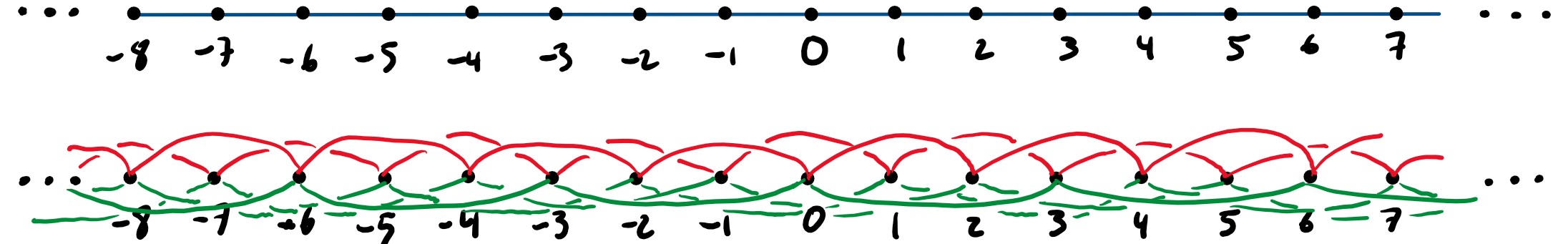}}
\caption{$Cay(\integers,S)$ para $S=\{1\}$ y para $S=\{2,3\}$ (estas son gr\'aficas dirigidas, pero omitimos las orientaciones en el dibujo por simplicidad).}
\label{fig:enteros}
\end{figure}

\begin{figure}[h!]
\centerline{\includegraphics[scale=0.5]{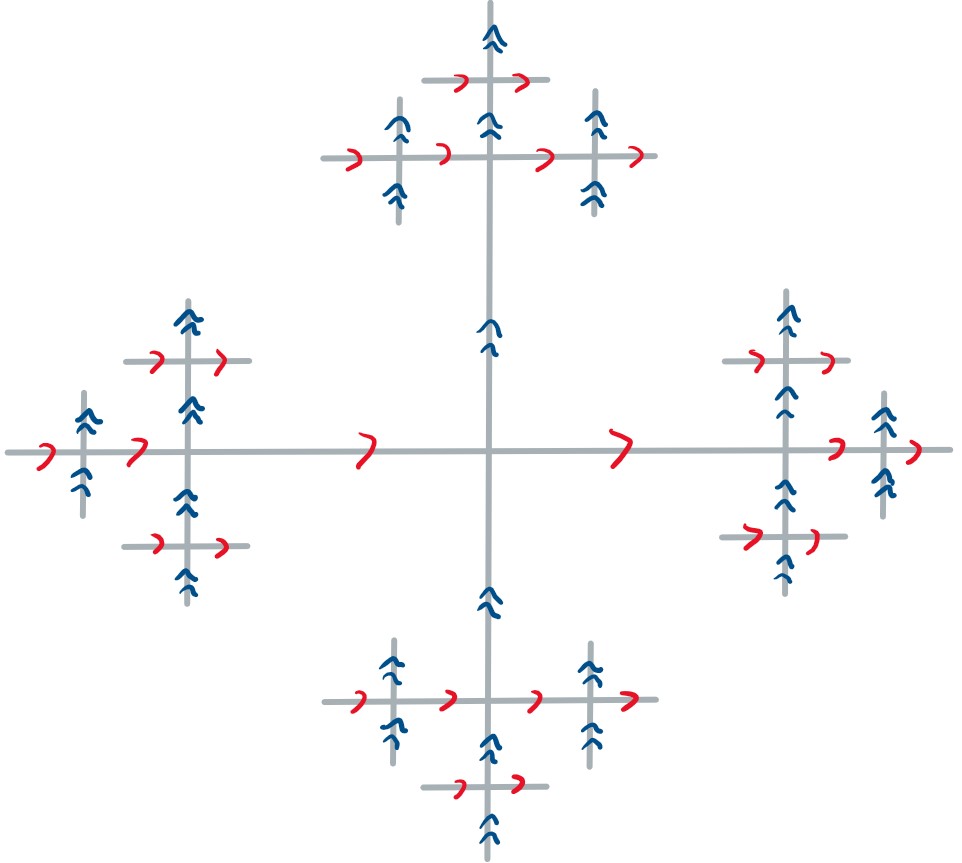}}
\caption{$Cay(F_2,S)$ para $S=\{a,b\}$. Para no atascar el dibujo, usamos las flechas horizontales para denotar a la `direcci\'on' de $a$ (o sea, los elementos de la forma $a^{\pm n}$ y  sus trasladados por las potencias de $b$) y  las flechas verticales para denotar a la direcci\'on de $b$ (o sea, los elementos de la forma $b^{\pm n}$ y sus trasladados por las potencias de $a$).}
\label{fig:free}
\end{figure} 

\begin{definition}\label{def:finetriangle}
Sea $X$ un espacio m\'etrico geod\'esico, un tri\'angulo $\Delta$ en $X$ es \emph{$\delta$-fino} para $\delta \geq 0$ si cada lado de $\Delta$ est\'a contenido en la uni\'on de las $\delta$-vecindades de los otros dos lados.
\end{definition}

\begin{definition}\label{def:hypgroup}
Un grupo finitamente generado $G$ es \emph{hiperb\'olico} si para alg\'un conjunto finito de generadores $S$ se satisface que todo tri\'angulo geod\'esico en  $Cay(G,S)$ es $\delta$-fino. 
\end{definition}

 Aunque la definici\'on de hiperbolicidad depende de $Cay(G,S)$, la propiedad de `ser hiperb\'olico' no depende del conjunto finito de generadores:
 
 \begin{proposition}
 Si $G$ es $\delta$-hiperb\'olico con respecto a $S$ entonces para todo conjunto finito de generadores $S'$, existe $\delta'$ para la cual $G$ es $\delta'$-hiperb\'olico con respecto a $S'$.
 \end{proposition} 
 
 En vista de la proposici\'on anterior, tiene sentido hablar de grupos hiperb\'olicos, y de hiperbolicidad en general, sin hacer referencia expl\'icita a un conjunto de generadores ni a una $\delta$ espec\'ifica.

La noci\'on de hiperbolicidad para grupos se le debe a Gromov~\cite{Gromov87}. Por esta raz\'on, es posible que el lector encuentre que en la literatura a veces se le llama \emph{Gromov-hiperbolicidad} a la propiedad en la definici\'on~\ref{def:hypgroup}. Se\~nalamos adem\'as que  aunque en este texto solamente definimos hiperbolicidad para grupos (o equivalentemente, para sus gr\'aficas de Cayley), la definici\'on tambi\'en tiene sentido (y resulta muy \'util) para espacios m\'etricos geod\'esicos en general.

\begin{example} Los grupos libres finitamente generados son hiperb\'olicos (con $\delta=0$). Los grupos fundamentales de superficies cerradas con caracter\'istica de Euler negativa tambi\'en son hiperb\'olicos~\cite{MaxDehnBook}. Los grupos fundamentales de variedades compactas cerradas que admiten una m\'etrica de Riemann negativamente curvada tambi\'en son hiperb\'olicos.
\end{example}

\begin{figure}[h!]
\centerline{\includegraphics[scale=0.5]{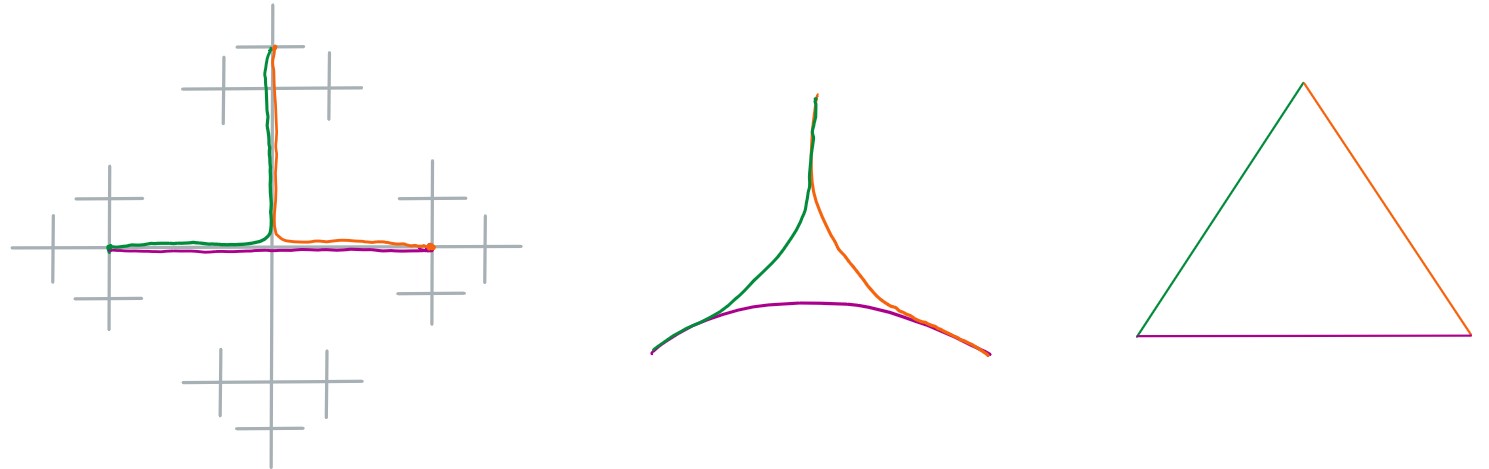}}
\caption{Ejemplos gen\'ericos de tri\'angulos geod\'esicos en $Cay(F_2,\{a,b\})$, en $\mathbb{H}^2$ y en  $\mathbb{E}^2$.}
\label{fig:hyperbolic}
\end{figure}

Otra clase amplia de ejemplos de grupos hiperb\'olicos surge de la noci\'on de cancelaciones peque\~nas; para definir los grupos de cancelaciones peque\~nas de manera rigurosa, necesitamos algunas nociones preliminares. 

Sea $\mathcal{P}=\langle S |R \rangle$ una presentaci\'on y sean $r,r' \in R$.  Si $r$ y $r'$ se pueden escribir de manera no trivial como concatenaciones $r=uv$ y $r'=uw$ decimos que $u$ es una \emph{pieza de $r$} (n\'otese que en este caso $u$ tambi\'en es una pieza de $r'$). 
Decimos que un conjunto de relaciones $R$ est\'a \emph{simetrizado} si $R$ es cerrado bajo tomar permutaciones c\'iclicas e inversos. Un conjunto de relaciones $R$ est\'a \emph{c\'iclicamente reducido} si para todo $r \in R$, si $r$ se puede expresar como una concatenaci\'on $r=wabw'$ o $r=awb$ entonces se satisface que $a \neq b^{-1}$.

\begin{definition} 
Decimos que una presentaci\'on c\'iclicamente reducida y simetrizada $\mathcal{P}$ \emph{satisface la condici\'on $C'(\frac{1}{n})$ de cancelaciones peque\~nas} si se satisface lo siguiente: si $u$ es una pieza de $r$, entonces $|u| < \frac{1}{n}|r|$. Decimos que $\mathcal{P}$  \emph{satisface la condici\'on $C(n)$ de cancelaciones peque\~nas} si siempre que $r$ se puede expresar como una concatenaci\'on de piezas $r=u_1\ldots u_k$, entonces $k \geq n$.
Decimos que un grupo $G$ satisface $C'(\frac{1}{n})$ o $C(n)$ si admite una presentaci\'on que satisface $C'(\frac{1}{n})$ o $C(n)$ respectivamente.
\end{definition}

Para mantener la notaci\'on compacta, la convenci\'on usual es que cuando hablamos de cancelaciones peque\~nas en una presentaci\'on, siempre asumimos que esta presentaci\'on est\'a simetrizada: no escribimos las relaciones `redundantes', o sea, las que se pueden deducir de otras relaciones.

La condici\'on $C'(\frac{1}{n})$ com\'unmente se conoce como la condici\'on \emph{m\'etrica} de cancelaciones peque\~nas  y la condici\'on $C(n)$ com\'unmente se conoce como la condici\'on \emph{no m\'etrica} de cancelaciones peque\~nas. Cabe se\~nalar que si una presentaci\'on $\mathcal{P}$ satisface $C'(\frac{1}{n})$ entonces satisface $C(n+1)$, pero que el inverso no se cumple: para todo $n \in \naturals$ y para todo $n >1$ existen presentaciones que satisfacen $C(n)$ pero no $C'(\frac{1}{n'})$. Encontrar ejemplos concretos queda como un ejercicio para el lector.

\begin{example}\label{ex:freecanc}
Como no tienen relaciones en sus presentaciones usuales, los grupos libres autom\'aticamente satisfacen las condiciones $C'(\frac{1}{n})$ y $C(n)$ para todo $n \in \naturals$.
\end{example}

\begin{example}\label{ex:surface}
Dada una superficie cerrada de g\'enero $g \geq 2$, la presentaci\'on `usual' de su grupo fundamental $\langle a_{1},b_{1},\dots ,a_{g},b_{g}\mid [a_{1},b_{1}]\cdot \dots \cdot [a_{g},b_{g}]\rangle$ satisface la condici\'on $C'(\frac{1}{7})$ de cancelaciones peque\~nas, porque todas las piezas en la presentaci\'on tienen longitud $1$. Si $g=1$, esta es la presentaci\'on usual del grupo fundamental del toro, y solamente satisface la condici\'on $C'(\frac{1}{3})$ de cancelaciones peque\~nas.
\end{example}

\begin{theorem}Los grupos finitamente presentados que satisfacen la condici\'on de cancelaciones peque\~nas $C'(\frac{1}{6})$ o $C(7)$  son hiperb\'olicos.
\end{theorem}

La siguiente noci\'on, aunque central en teor\'ia geom\'etrica de grupos, no ser\'a utilizada de manera expl\'icita en lo que sigue del texto. Consideramos pertinente inclu\'irla para darle m\'as contexto al lector, y por su relevancia en algunos ejemplos.

\begin{definition}\label{def:qi}
Sean $(X_1, d_1), (X_2, d_2)$ espacios m\'etricos. Una funci\'on $f : X_1 \rightarrow X_2$ es un \emph{$(A,B)$-encaje cuasi-isom\'etrico}  si existen constantes $A \geq 1$ y $B \geq 0$ tales que para cualesquiera $x, y \in X_1$ se tiene:
$$\frac{1}{A}d_1(x, y) - B \leq d_2(f(x), f(y)) \leq Ad_1(x, y) + B.$$
Si $f : X_1 \rightarrow X_2$ es un encaje cuasi-isom\'etrico, decimos que $f$ es una \emph{$(A,B)$-cuasi-isometr\'ia} si existe una constante $C>0$ con la propiedad de que para todo $x_2 \in X_2$ existe $x_1 \in X_1$ tal que $d_2(f(x_1), x_2) < C$.
\end{definition}

\begin{remark} La hiperbolicidad en el sentido de la definici\'on~\ref{def:hypgroup} es un invariante cuasi-isom\'etrico: es decir, si $Cay(G,S)$ y $Cay(G',S')$ son gr\'aficas de Cayley cuasi-isom\'etricas y $Cay(G,S)$ es hiperb\'olico, entonces $Cay(G',S')$ tambi\'en lo es.
\end{remark}

Como en el caso de la definici\'on de hiperbolicidad, al hablar de cuasi-isometr\'ias y encajes cuasi-isom\'etricos omitimos hacer menci\'on de las constantes a menos de que estas sean utilizadas de manera expl\'icita. 

\begin{definition}
Un subgrupo $H$ de un grupo $G$ es \emph{cuasiconvexo} si existe $m > 0$ tal que lo siguiente se cumple para todas las geod\'esicas $\sigma$ en  $Cay(G,S)$: si los extremos de $\sigma$ yacen en $H$ entonces $\sigma \subset N_m(H)$.
\end{definition}

Para un subgrupo $H$, ser cuasiconvexo, en general, depende de la elecci\'on de los generadores. Sin embargo, cuando $G$ es hiperb\'olico, ser cuasiconvexo es un invariante del subgrupo, y es equivalente a ser finitamente generado y estar cuasi-isometricamente encajado. Constatamos esto en la siguiente proposici\'on:

\begin{proposition}
Si $G$ es hiperb\'olico y $H$ es un subgrupo de $G$ que es cuasiconvexo en $Cay(G,S)$, entonces $H$ es cuasiconvexo en $Cay(G,S')$ para todo conjunto finito de generadores $S' \subset G$.  
\end{proposition}

Una de las propiedades principales de las isometr\'ias locales es que son $\pi_1$-inyectivas. Es natural preguntarse entonces \textbf{cu\'ales} subgrupos del grupo fundamental de un complejo cubular npc $X$ se pueden realizar por isometr\'ias locales donde el dominio es un complejo cubular compacto. Cuando $\pi_1 X$ es hiperb\'olico, una clase amplia de subgrupos teniendo esta propiedad son los subgrupos cuasiconvexos:

\begin{proposition}[\cite{HaglundGraphProduct}]\label{cor:loci}
Sea $X$ un complejo cubular compacto npc cuyo grupo fundamental es hiperb\'olico. Sea $H < \pi_1X$ un subgrupo cuasiconvexo. Entonces existe una isometr\'ia local $Y \rightarrow X$ donde $Y$ es compacto y $\pi_1 Y = H$. 
\end{proposition}

\section{Espacios de paredes, o de c\'omo cubular grupos.}

\subsection{Acciones en complejos CAT(0)}

\begin{definition}
Decimos que un grupo $G$ act\'ua \emph{propiamente} en un espacio m\'etrico $X$ si el conjunto $\{g | gK \cap K\}$ es finito para todo subespacio compacto $K \subset X$. Decimos adem\'as que $G$ act\'ua de manera \emph{m\'etricamente propia} en un espacio m\'etrico $X$ si para cada $r>0$ y para cada $x \in X$ el conjunto $\{g | gB_r(x)\cap B_r(x) \neq \emptyset\}$ es finito. Finalmente, decimos que $G$ act\'ua  \emph{cocompactamente} en $X$ si  existe un subespacio compacto $K \subset X$ tal que
${\displaystyle X=G\cdot K}$.
\end{definition}

 Si una acci\'on es cocompacta, actuar propiamente y actuar de manera m\'etricamente propia son nociones equivalentes, pero este no es el caso m\'as en general.
 Por otra parte, si una acci\'on es propia, ser cocompacto es equivalente a que el espacio cociente ${\displaystyle G\backslash X}$ sea compacto.

\begin{definition}
Un grupo $G$ est\'a \emph{cubulado} si act\'ua de manera m\'etricamente propia en un complejo cubular CAT(0). Si la acci\'on es adem\'as cocompacta, decimos que $G$ est\'a \emph{cocompactamente cubulado}.
\end{definition}

\begin{remark} En este art\'iculo (y en general en teor\'ia geom\'etrica de grupos), cuando un grupo act\'ua en un espacio m\'etrico, \textbf{\textsf{la acci\'on es siempre por isometr\'ias}} a menos de que se especifique lo contrario.
\end{remark}

Algunos ejemplos de grupos cubulados son los siguientes:

\begin{example} Si $X$ es un complejo cubular npc, entonces $\pi_1 X \curvearrowright \widetilde X$ de manera m\'etricamente propia (y cocompactamente si $X$ es finito), as\'i que $\pi_1 X$ est\'a cubulado (y cocompactamente cubulado si $X$ es finito).
\end{example}

\begin{example} Si $G$ satisface la condici\'on $C'(\frac{1}{6})$ de cancelaciones peque\~nas, entonces $G$ est\'a cocompactamente cubulado~\cite{WiseSmallCanCube04}.
\end{example}

\begin{example} Si $G$ es un grupo de Coxeter (ver el ejemplo~\ref{ex:coxeter}) entonces $G$ est\'a  cubulado~\cite{NibloRoller98}.
\end{example}

\begin{example} Si $M$ es una 3-variedad hiperb\'olica (y para muchas otras 3-variedades!), entonces $\pi_1 M$ est\'a cubulado.
\end{example}

\begin{remark}\label{rmk:torsioncc} Los grupos fundamentales de complejos cubulares npc (de dimensi\'on finita o infinita) no tienen torsi\'on. Un grupo con torsi\'on s\'i puede estar cubulado, pero hay algunas restricciones: por ejemplo, si $G$ es finitamente generado y act\'ua en un complejo cubular CAT(0) de dimensi\'on finita sin puntos fijos globales (no existe $p \in \widetilde X$ tal que $gp=p$ para todo $g \in G$), entonces $G$ no es de torsi\'on (esto es, $G$ contiene elementos de orden infinito).
\end{remark}

En general, \textbf{¿c\'omo hacemos para cubular un grupo?} Y si un grupo no se puede cubular, \textbf{¿c\'omo hacemos para demostrarlo?}
En lo que sigue abordaremos la primera de estas preguntas; la segunda es una direcci\'on un tanto diferente, y va m\'as all\'a de los objetivos de este art\'iculo.

\subsection{Espacios de paredes}

\begin{definition}[\cite{HaglundPaulin98}]\label{def:walls}
Un \emph{espacio de paredes} es un conjunto $S$ junto con una colecci\'on  $\{W_i\}_{i \in I}=\mathcal{W}$ donde $W_i=(\overleftarrow W_i,\overrightarrow W_i)$ y tal que:
\begin{enumerate}
    \item \label{cond:1w} $\overleftarrow W_i\cup \overrightarrow W_i=S$ y
    \item \label{cond:2w} $\overleftarrow W_i\cap \overrightarrow W_i=\emptyset$.
\end{enumerate}
Adem\'as, $\mathcal{W}$ satisface que, para todo par $p,q \in S$ el n\'umero de paredes que separa a $p$ de $q$, denotado por $\#_{\mathcal{W}}(p,q)$, es finito. Los pares  $W_i=(\overleftarrow W_i,\overrightarrow W_i)$ son las  \emph{paredes} de $(S,\mathcal{W})$, y  los $\overleftarrow W_i, \overrightarrow W_i$ son los \emph{semiespacios} de la pared $W_i$.
\end{definition}

\begin{remark} La condici\'on de que los semiespacios asociados a una pared $W_i$ sean ajenos se puede debilitar un poco: en la pr\'actica, basta con pedir que para todo punto $p \in S$, el conjunto $\{ W_i | p \in \overleftarrow W_i\cap \overrightarrow W_i \}$ tenga cardinalidad finita.
\end{remark}

\begin{figure}[h!]
\centerline{\includegraphics[scale=0.48]{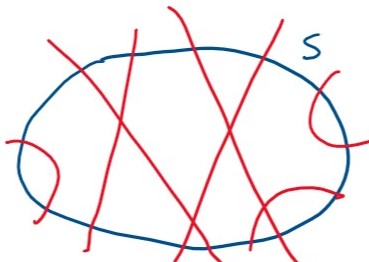}}
\caption{Idea de la definici\'on de espacio de paredes.}
\label{fig:ex0wall}
\end{figure} 

\begin{example}
Sea $S=\reals^2$ y $\mathcal{W}=\{\text{l\'ineas en el plano}\}$, entonces $\mathcal{W}$ satisface las condiciones (1) y (2), pero para todo $p\neq q \in \reals^2$, la cantidad de l\'ineas que separan a $p$ de $q$ es no-numerable. Este tambi\'en es el caso para $\reals^n$ para cualquier  $n$ cuando $\mathcal{W}$ es el conjunto de todos los subespacios afines de codimensi\'on 1 en $\reals^n$.
\end{example}

\begin{example}
Sea $S=\reals^2$ y sean $\ell, \ell'$ dos l\'ineas ortogonales, sea  $\mathcal{W}$ el conjunto de trasladados enteros de $\ell$ y $\ell'$. Entonces $(\reals^2,\mathcal{W})$ es un espacio de paredes. De manera similar, $(\reals^n, \mathcal{W})$ es un espacio de paredes para todo $n$ tomando $\mathcal{W}$ como el conjunto de trasladados enteros de cualquier par de subespacios afines ortogonales de codimensi\'on $1$.  
\end{example}

\begin{example}\label{ex:coxeter}
 Los \emph{grupos de Coxeter} est\'an dados por presentaciones de la forma $\langle a_{1},a_{2},\ldots ,a_{n}\mid (a_{i}a_{j})^{m_{ij}}=1\rangle $
donde $m_{ii}=1$ y $m_{ij}\geq 2$ si $i\neq j$ y donde definimos $m_{ij}= \infty$ si no hay relaci\'on entre $a_i$ y $a_j$. N\'otese que las relaciones de la forma $m_{ii}=1$ implican que el grupo correspondiente est\'a generado por elementos de orden $2$.  As\'i, todo  grupo de Coxeter se puede ver como un grupo generado por reflecciones en un espacio adecuado. Estos grupos tienen una estructura natural de espacio de paredes, donde $S$ son los elementos del grupo y las paredes est\'an dadas por las reflecciones que generan al grupo y por sus `trasladados'. Ver la figura~\ref{fig:morewalls} para un ejemplo sencillo.
\end{example}

La situaci\'on que m\'as nos interesa es la siguiente:

\begin{example} Para todo complejo cubular CAT(0) $\widetilde X$, su $0$-esqueleto $\widetilde X^{(0)}$ es un espacio de paredes, donde la estructura de espacio de paredes viene de tomar $\mathcal{W}=\{\text{hiperplanos de }\widetilde X\}$. Ver la figura~\ref{fig:morewalls}.
\end{example}

\begin{figure}[h!]
\centerline{\includegraphics[scale=0.533]{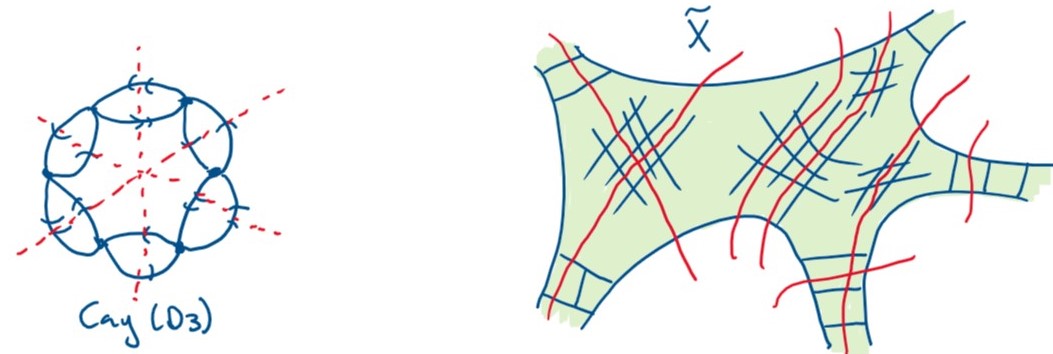}}
\caption{Espacios de paredes asociados a un grupo de Coxeter, en este caso $D_3$, y a un complejo cubular CAT(0).}
\label{fig:morewalls}
\end{figure} 

Lo siguiente que nos toca hacer es entender c\'omo puede actuar un grupo $G$ en un conjunto $S$ de tal manera que preserve la estructura de espacio de paredes dada.

\begin{definition} Un grupo $G$ \emph{act\'ua en un espacio de paredes} $(S, \mathcal{W})$ si act\'ua en $S$ y permuta las paredes, es decir, para todo $W_i \in \mathcal{W}$ y para todo $g \in G$, la pareja $(g\overleftarrow W_i,g\overrightarrow W_i)\in \mathcal{W}$.
\end{definition}

\subsubsection{La construcci\'on de Sageev}

 A  continuaci\'on describiremos una construcci\'on, introducida por Sageev\footnote{Se pronuncia \emph{saguiv}, no \emph{sagueiev}.} en~\cite{Sageev95} que produce, dado un espacio de paredes arbitrario, un complejo cubular CAT$(0)$  $C$ llamado el  \emph{dual} de $(S,\mathcal{W})$. 
Cada 0-cubo $v \in C^{(0)}$ consiste de una elecci\'on de un semiespacio para cada pared $W_i$, en otras palabras un v\'ertice $v$ corresponde a un conjunto $v=\{\bar W_i\}_{i \in I}$ donde cada $\bar W_i \in \{\overleftarrow W_i,\overrightarrow W_i\}$. Pedimos adem\'as que estas elecciones satisfagan las siguientes propiedades:
\begin{enumerate}
    \item \label{cd:w1} Toda pareja de semiespacios en $v$ se interseca,
    \item \label{cd:w2} para todo $x \in S$, a lo m\'as un n\'umero finito de semiespacios pertenecientes a un mismo $v$ no contiene a $x$.
\end{enumerate}

El complejo $C$ tiene un 1-cubo conectando 0-cubos $u,v$ si y solo si $u$ y $v$ difieren en una \'unica pared. Los cubos de dimensi\'on m\'as alta se a\~naden de manera inductiva: para cada $n\geq 2$, hay un $n$-cubo siempre que su $(n-1)$-esqueleto est\'a presente.

N\'otese que las paredes en $S$ est\'an en correspondencia biyectiva con los hiperplanos de $C$.

\begin{figure}[h!]
\centerline{\includegraphics[scale=0.52]{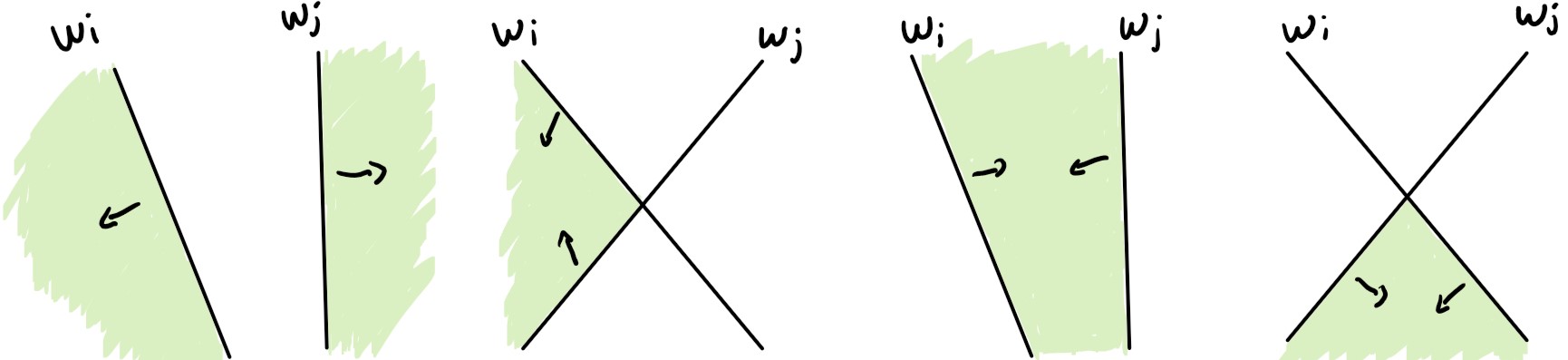}}
\caption{Idea de la condici\'on~\eqref{cd:w1} en la construcci\'on de Sageev: la elecci\'on de semiespacios (sombreada) en la izquierda no puede corresponder a un v\'ertice de $C$; las tres elecciones de la derecha s\'i.}
\label{fig:paredes1}
\end{figure}

\begin{figure}[h!]
\centerline{\includegraphics[scale=0.53]{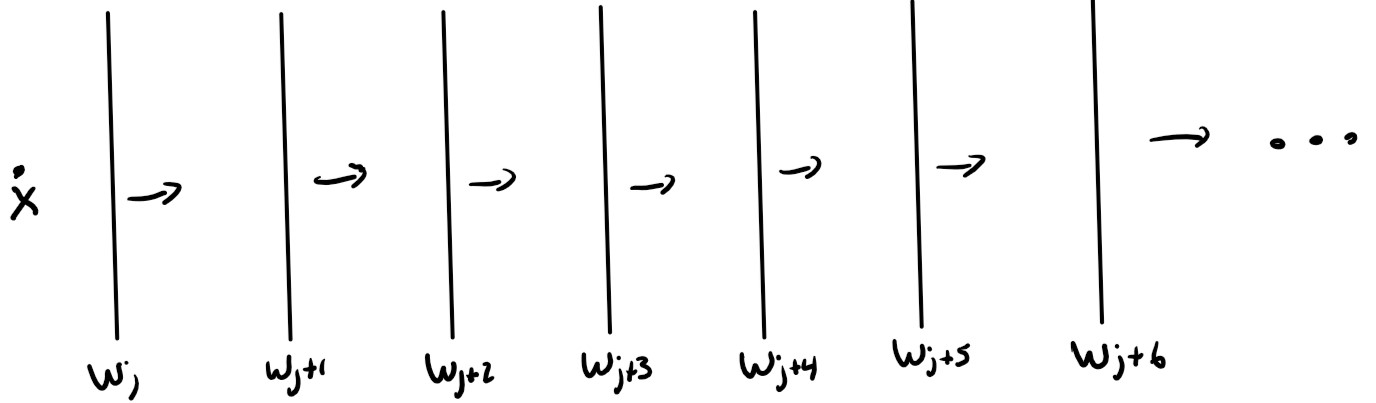}}
\caption{Ejemplo de c\'omo puede fallar la condici\'on~\eqref{cd:w2} en la construcci\'on de Sageev.}
\label{fig:paredes2}
\end{figure}

Dos paredes $W_i, W_j$ se \emph{cruzan} si los cuatro conjuntos $\overleftarrow W_i \cap \overleftarrow W_j, \overleftarrow W_i \cap \overrightarrow W_j, \overrightarrow W_i \cap \overrightarrow W_j, \overrightarrow W_i \cap \overleftarrow W_j $ son no vacios. Cuando ning\'un par de paredes en $(S, \mathcal{W})$ se cruza, el complejo cubular dual $C$ es un \'arbol.

\begin{figure}[h!]
\centerline{\includegraphics[scale=0.385]{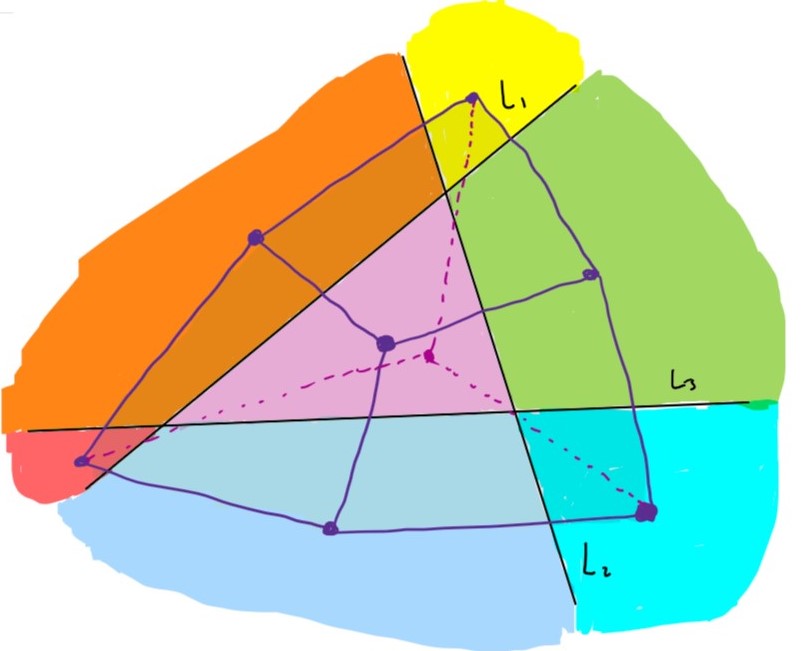}}
\caption{El espacio de paredes $(\reals^2, \mathcal{W})$ donde $\mathcal{W}$ consiste de 3 l\'ineas $L_1,L_2, L_3$ que se intersecan por pares, y su complejo cubular dual.}
\label{fig:3lines}
\end{figure}

\begin{remark}
Sean $x,x' \in X$, entonces $v_x=v_{x'}$ siempre que $\#_{\mathcal{W}(x,x')}=0$.
\end{remark}

\begin{theorem}[\cite{Sageev95}]  El complejo cubular $C$ dual a un espacio de paredes $(S,\mathcal{W})$ es CAT(0).
\end{theorem}

\begin{proof}
La idea de la demostraci\'on no es dificil, pero s\'i es un poquito verbosa: hay que demostrar que $C$ es npc y simplemente conexo, primero vamos a ver que la aureola $link(v)$ es simplicial para cada $v \in C^{(0)}$, luego que cumple la condici\'on de ser bandera, y finalmente que $C$ es simplemente conexo.

Sea $v \in C^{(0)}$. Supongamos que $link(v)$ no es simplicial -- o sea, por ejemplo, que hay un b\'igono en $link(v)$. Esto implica que hay dos cuadrados $c$ y $c'$ en $C$ que se pegan a lo largo de dos 1-cubos consecutivos. Sean $u,v,w,z$ y $u,v,w,z'$ los v\'ertices de $c$ y $c'$ respectivamente. Ahora, cada uno de estos v\'ertices corresponde a una elecci\'on de semi-espacios de $\mathcal{W}$. Los v\'ertices $z$ y $z'$ corresponden a elecciones que difieren en un \'unico semiespacio de las elecciones correspondientes a $u$ y $w$, respectivamente, pero solamente hay dos elecciones posibles de semiespacios con esta propiedad, y uno de ellos tiene que corresponder a $v$. Esto implica que $z=z'$ y que por lo tanto $c=c'$. El argumento para ver que no hay lazos en $link(v)$ es similar.

Ahora hay que checar que $link(v)$ es bandera. Basta verificar que si hay ciclos de longitud 3 en $link(v)$, entonces estos bordean 2-simplejos. Si hay un ciclo de longitud 3 en $link(v)$, entonces hay 3 cuadrados que se pegan alrededor de $v$ y hay 3 paredes correspondientes a los hiperplanos que son duales a las 3 aristas que emanan de $v$. Pero vimos en uno de los ejemplos que el complejo cubular dual a una colecci\'on de 3 paredes que se cruzan todas por pares es un 3-cubo, o sea que los 3 cuadrados que se pegan alrededor de $v$ en realidad forman la esquina de un 3-cubo. Como este cubo est\'a presente en $C$ el 3-ciclo en $link(v)$ bordea un 2-simplejo.

Finalmente, hay que mostrar que $C$ es simplemente conexo. Esto queda como ejercicio.
\end{proof}

\begin{remark}
Sea $(S, \mathcal{W})$ un espacio de paredes y $\mathcal{C}$ su complejo cubular dual. Un cubo maximal de dimensi\'on $n$ en $\mathcal{C}$ corresponde a una colecci\'on maximal de $n$ paredes que se cruzan todas por pares (ejercicio). 
\end{remark}

\subsubsection{Acciones en espacios de paredes}

\begin{lemma}\label{lem:wallacts}Si $G$ act\'ua en un espacio de paredes $(S,\mathcal{W})$, entonces $G$ act\'ua en su dual $C$.
\end{lemma}

\begin{proof}
B\'asicamente se sigue de la definici\'on: 
La acci\'on de  $g \in G$ en  $(X,\mathcal{W})$ env\'ia a $W$ a una pared $gW \in \mathcal{W}$ y a los semiespacios correspondientes $\{\overleftarrow W, \overrightarrow W\}$ a semiespacios $\{g\overleftarrow W, g\overrightarrow W\}$,  as\'i que cada $g \in G$ act\'ua en el v\'ertice $v \in C^{(0)}$ dado por una elecci\'on de semiespacio para cada pared $W$ en $\mathcal{W}$ mand\'andolo al v\'ertice $v'=:gv$ donde los semiespacios que determinan $v'$ corresponden a los trasladados bajo $g$. Cada arista  $e$ de $C$ corresponde a un par de elecciones que difieren en una \'unica pared, as\'i que la acci\'on se extiende de  $C^{(0)}$ a $C^{(1)}$ y por lo tanto, inductivamente, a todo $C$. 
\end{proof}

Idealmente, nos gustar\'ia que las acciones en complejos cubulares fueran propias y cocompactas para poder decir que el grupo que est\'a actuando est\'a cocompactamente cubulado. Hay varios criterios que sirven para esto.

\begin{theorem}
$G$ act\'ua cocompactamente en $C$ si y solo si  hay un n\'umero finito de $G$-\'orbitas de colecciones maximales de paredes que se cruzan todas por pares.
\end{theorem}

Para determinar si las acciones son m\'etricamente propias, un criterio es el siguiente:

\begin{theorem}
Si $\#(p,gp)\rightarrow \infty$ cuando $\dist(1,g)\rightarrow \infty $ entonces $G$ act\'ua de manera m\'etricamente propia en $C$.
\end{theorem}

\subsubsection{Subgrupos de codimensi\'on 1}

Una situaci\'on m\'as algebraica en donde tenemos una estructura de espacio de paredes es cuando $G$ es finitamente generado y $H$ es  un subgrupo que  `separa' a la gr\'afica de Cayley  $Cay(G,S)$ de $G$. 

Dado un subgrupo $H<G$ usamos la notaci\'on $\mathcal{N}_r(H)=\{x:d(x,h)\leq r \text{ para alg\'un } h \in H \}$ para denotar la $r$-vecindad de $H$ en $Cay(G,S)$.

\begin{definition}
Una componente conexa  $A$ de $Cay(G,S)- \mathcal{N}_r(H)$ es \emph{profunda} si para todo $s > 0$ tenemos que $A \nsubseteq \mathcal{N}_s(H)$.
Un subgrupo $H<G$ tiene \emph{codimensi\'on 1} si para todo $r > 0$ $\mathcal{N}_r(H)$  separa a $Cay(G,S)$ en al menos dos \'orbitas de componentes conexas profundas.
\end{definition}

Una componente conexa $A \subset G$ y su complemento $A'$ forman una pared $\mathcal{A}$.  La pareja $(G, \mathcal{W})$ con  $\mathcal{W}=\{g\mathcal{A}| g \in G\}$ es un espacio de paredes.

Aunque no es cierto en general que un subgrupo de codimensi\'on 1 tenga `dimensi\'on 1 menos que el grupo original' en ning\'un sentido formal, esta terminolog\'ia est\'a justificada por los ejemplos que surgen de la topolog\'ia en dimensiones bajas: por ejemplo, en el grupo fundamental de una superficie cerrada y orientable, un subgrupo c\'iclico de codimensi\'on 1 corresponde a una curva cerrada no nullhomot\'opica en la superficie, y de manera an\'aloga, en el grupo fundamental de una 3-variedad orientable cerrada, el grupo fundamental de una superficie inmersa incompresible es un subgrupo de codimensi\'on 1.

\begin{remark} Tener codimensi\'on 1 es independiente de la elecci\'on de generadores para $G$. 
\end{remark}

\begin{example} Todo subgrupo $\mathbb{Z}^{n-1} \hookrightarrow \mathbb{Z}^n$ es de codimensi\'on 1.
\end{example}

Cuando $G$ es hiperb\'olico, sus subgrupos cuasiconvexos pueden servir para cubular (comparar con la proposici\'on~\ref{cor:loci}):

\begin{proposition}[\cite{Sageev97}] \label{prop:coco}Si $G$ es hiperb\'olico y los estabilizadores de las paredes corresponden a una colecci\'on $\{H_1, \ldots, H_n\}$  de subgrupos cuasiconvexos, entonces la acci\'on de $G$ en su complejo cubular dual $C$ es cocompacta. Si, adem\'as, para cada elemento de orden infinito $g \in G$, su eje $E_g$ es cortado por alg\'un trasladado $g'H_j$, entonces la acci\'on  es propia.
\end{proposition}

\section{Complejos (virtualmente) especiales}

Las nociones y los resultados presentados en esta secci\'on se originan en~\cite{HaglundWiseSpecial}.

Comenzamos con los siguientes cuatro comportamientos  que se pueden presentar entre hiperplanos en un complejo cubular:

\textbf{Orientabilidad:} Un hiperplano $H$ en $X$ \emph{tiene dos lados} si todos los 1-cubos duales a $H$ se pueden orientar de manera consistente. 

\textbf{Cruces:} Dos hiperplanos  $H$ y $H'$ \emph{se cruzan} si tienen cubos centrales pertenecientes a un mismo cubo de $X$. Un hiperplano est\'a \emph{encajado} si no se cruza a si mismo.

\textbf{Autoosculaci\'on:} Un hiperplano $U$ se \emph{autooscula} si es dual a dos o m\'as 1-cubos dirigidos que comparten el mismo v\'ertice inicial o final.

\textbf{Interosculaci\'on:} Dos hiperplanos $U,V$ se  \emph{interosculan} en $p \in X^{(0)}$ si se cruzan y son duales a 1-cubos $e,f$ en $p$ que no se encuentran en la esquina de un 2-cubo (o sea, $e$ y $f$ no son adyacentes en $link(p)$).  

\begin{definition}
Un complejo cubular npc $X$ es \emph{especial} si las siguientes condiciones se cumplen:

\begin{enumerate}
\item \label{en:sp1} Todos los hiperplanos de $X$ tienen dos lados.
\item \label{en:sp2} Todos los hiperplanos de $X$ est\'an encajados.
\item \label{en:sp3} Ning\'un hiperplano de $X$ se autooscula.
\item \label{en:sp4} Ning\'un par de hiperplanos de $X$ se interoscula.
\end{enumerate}
\end{definition}

\begin{figure}[h!]
\centerline{\includegraphics[scale=0.45]{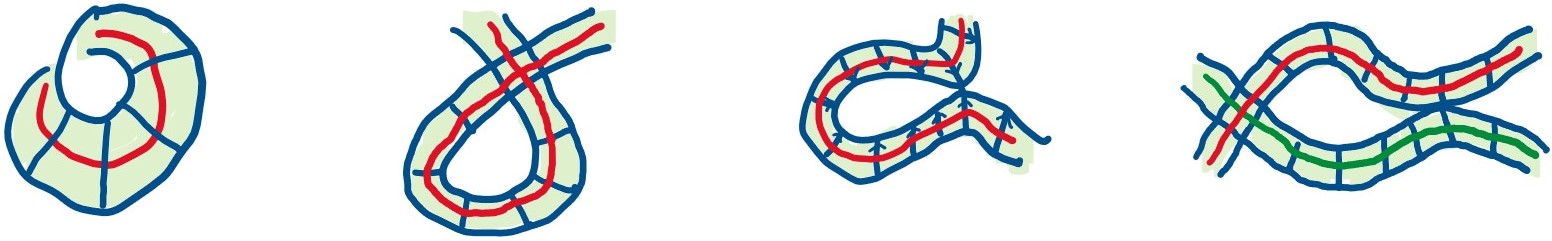}}
\caption{Las cuatro posibles patolog\'ias.}
\label{fig:pathologies}
\end{figure}

\begin{figure}[h!]
\centerline{\includegraphics[scale=0.45]{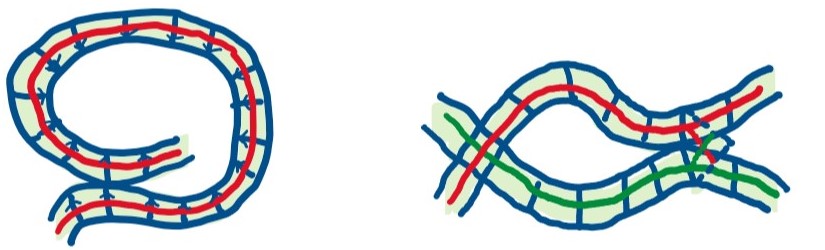}}
\caption{Dos situaciones que no constituyen patolog\'ias: en la primera no hay autoosculaci\'on, porque los 1-cubos que comparten el mismo  0-cubo est\'an dirigidos de manera consistente; en la segunda, no hay interosculaci\'on, pues los hiperplanos se intersecan en el cuadrado que est\'a `detr\'as' del resto.}
\label{fig:nonpathologies}
\end{figure}

\begin{example} Toda gr\'afica es especial: como todos los hiperplanos son baricentros de aristas, las condiciones se satisfacen automaticamente.
\end{example}

\begin{example} Los complejos cubulares CAT(0) son especiales. Esto se sigue del teorema~\ref{thm:cat0}: la conclusi\'on~\eqref{it:cat3} implica que los hiperplanos tienen dos lados, la conclusi\'on~\eqref{it:cat1} implica que los hiperplanos est\'an encajados, la conclusi\'on~\eqref{it:cat5} dice exactamente que los hiperplanos no se interosculan, y la conclusi\'on~\eqref{it:cat2} implica, junto con el teorema~\ref{thm:local}, que los hiperplanos no se autoosculan.
\end{example}

Otros ejemplos se pueden obtener de la siguiente proposici\'on:

\begin{proposition}\label{prop:specialprod} Sea $X \subset (A\times B)$ un subcomplejo del producto de dos gr\'aficas $A$ y $B$. Entonces $X$ es especial. 
\end{proposition}

\begin{proof}
 Empezamos por considerar $X = A\times B$. Ya sabemos que $A\times B$ es npc porque esto se cumple para cualquier producto de complejos cubulares npc. Para mostrar que  $A\times B$ es especial, analizamos los hiperplanos. Todo hiperplano en $A\times B$ es de la forma $H=A \times \{b\}$ o $H=\{a\} \times B$ donde $a$ es un hiperplano en $A$ y $b$ es un hiperplano en $B$. Como estos son productos, siempre tienen dos lados. Cada cuadrado de $A\times B$ es el producto de una arista en $A$ con una arista en $B$, as\'i que ning\'un hiperplano de $A \times B$ puede tener dos semicubos pertenecientes a un mismo cubo, y cada hiperplano tiene que estar encajado. Si un hiperplano $H$ se autooscula, entonces  hay al menos dos aristas dirigidas consistentemente que son duales a $H$ y  que comparten el mismo v\'ertice inicial o final, de nuevo esto contradice que $A \times B$ es un producto. Finalmente, si dos hiperplanos $H, H'$ se interosculan, entonces el cuadrado en el que se cruzan indica que, digamos, $H$ y $H'$ son de la forma $A \times \{b\}$ y $\{a\} \times B$, respectivamente. Pero el v\'ertice $p$ en donde los soportes de $H$ y $H'$ se tocan sin que $H$ y $H'$ se intersecten indica que los dos hiperplanos son de la forma $A \times \{b\}$ o $\{a\} \times B$, as\'i que de nuevo llegamos a una contradicci\'on.

Ahora, si un complejo cubular $Y$ de dimensi\'on $2$ es especial, entonces todo subcomplejo $Y'$ de $Y$ tambi\'en lo es. En efecto, $Y'$ es un complejo cubular no positivamente curvado pues $Y$ lo es, y para cada $v \subset Y'\subset Y$, la restricci\'on de $link(v)$ a $Y'$ no puede contener ciclos que no estuvieran ya presentes en $Y$. De manera similar, $Y'$ es especial porque toda patolog\'ia entre hiperplanos de $Y'$ tendr\'ia que provenir de una patolog\'ia entre hiperplanos de $Y$.  
\end{proof}

La primera parte de la demostraci\'on anterior se puede extender a productos arbitrarios de complejos cubulares especiales. 

\begin{theorem}[El Ejemplo]\label{thm:salvspe}
El complejo de Salvetti $R$ de un gaar es especial.
\end{theorem}

\begin{proof}[Sketch]
Primero notemos que cada hiperplano de $R$ es dual a un \'unico 1-cubo. Esto es porque los \'unicos 1-cubos en $R$ corresponden a los generadores del gaar.
Los hiperplanos tiene dos lados porque cada 2-cubo en $R$ corresponde a un conmutador en la presentaci\'on y est\'an encajados porque los conmutadores est\'an asociados a generadores distintos (o sea, no estamos etiquetando ning\'un cuadrado con un mismo generador en los cuatro lados y ning\'un par de generadores aparece en m\'as de un cuadrado).
Como cada hiperplano solo es dual a un 1-cubo, no puede haber ni autoosculaci\'on ni interosculaci\'on.
\end{proof}

Suele ser el caso que un complejo cubular npc $X$ falla en ser especial, pero tiene una cubierta de grado finito que s\'i lo es. Si esto ocurre, decimos que $X$ es \emph{virtualmente especial}. 

\begin{figure}[h!]
\centerline{\includegraphics[scale=0.6]{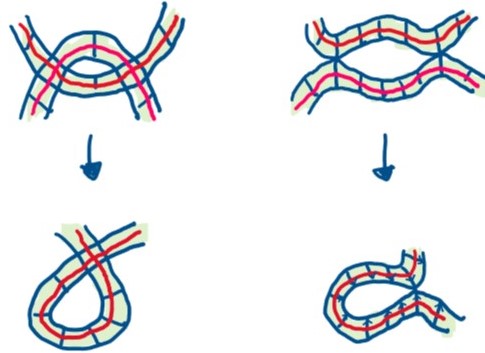}}
\caption{Dos maneras de levantar patolog\'ias a no-patolog\'ias en cubiertas de grado 2.}
\label{fig:special2}
\end{figure}

\begin{lemma}\label{lem:projects}Si $A \looparrowright B$ es una isometr\'ia local y $B$ es especial, entonces $A$ es especial. En particular, todo espacio cubriente $\widehat{B} \rightarrow B$ de un complejo cubular especial $B$ es tambi\'en especial.
\end{lemma}

\begin{proof}[Sketch]
Supongamos que $\varphi: A \rightarrow B$ es una isometr\'ia local.
Si $A$ no es especial, entonces alguna de las cuatro patolog\'ias descritas al principio de esta secci\'on ocurre en $A$. Si, por ejemplo, $A$ tiene un hiperplano $H$ que se autointerseca en un cubo $c$, entonces $\varphi(c)$ es un cubo en $B$ en donde se cruzan dos semicubos pertenecientes al hiperplano $\varphi(H)$, por lo que $B$ no puede ser especial; si, en cambio, hay un par de hiperplanos  $H, H'$ en $A$ que se interosculan en un v\'ertice $p$, entonces $H$ y $H'$ se cruzan y son duales a aristas $e,f$ en $p$ que no se encuentran en la esquina de un 2-cubo en $A$, por lo que los hiperplanos   $\varphi(H)$ y $\varphi(H')$ se cruzan y son duales a aristas $\varphi(e),\varphi(f)$ en $\varphi(p)$. Las aristas $\varphi(e),\varphi(f)$  no se pueden encontrar en la esquina de un 2-cubo en $B$, pues esto contradecir\'ia que $\varphi$  es una isometr\'ia local. As\'i que $\varphi(H)$ y $\varphi(H')$  seinterosculan, y  concluimos nuevamente que $B$ no puede ser especial.
Analizar las otras dos posibles patolog\'ias queda como ejercicio para el lector.
\end{proof}

\begin{theorem}\label{thm:specialvetti} Un complejo cubular npc $X$ es especial si y solo si existe una isometr\'ia local $X \rightarrow R$ donde $R$ es un complejo de Salvetti. 
\end{theorem}

\begin{proof}
Una implicaci\'on se sigue del lema~\ref{lem:projects}.
Para la otra implicaci\'on, mostraremos que $R=R(\Gamma)$ donde $\Gamma$  es la \emph{gr\'afica de incidencia} de $X$. Esto es, una gr\'afica simplicial con un v\'ertice $v_H$ por cada hiperplano $H$ de $X$ y donde existe una arista entre v\'ertices $v_H$ y $v_{H'}$ si y solo si los hiperplanos $H$ y $H'$ se cruzan.
Para definir la isometr\'ia local, vamos a etiquetar cada arista de $R$ con el hiperplano al que es dual, que a su vez est\'a asociado a un hiperplano de $X$. Orientamos ahora las aristas (en $X$) para que la frontera de cada cuadrado sea un conmutador, es decir, queremos que lados opuestos de cuadrados en $X$ est\'en orientados consistententemente. Esto se puede porque estamos asumiendo que $X$ es especial y por lo tanto los hiperplanos tienen dos lados.
Esto nos da un mapeo combinatorio $X^{(1)} \rightarrow R$ que podemos extender a  $X \rightarrow R$. Afirmamos ahora que la extensi\'on es una isometr\'ia local.
En efecto, $X \rightarrow R$ es una inmersi\'on porque los hiperplanos de $X$ no se autoosculan, y cumple la condici\'on de que las esquinas de cuadrados en la imagen vengan de esquinas de cuadrados en $X$ porque los hiperplanos de $X$ no se interosculan.
\end{proof}

\begin{corollary}\label{cor:espsalv} Un complejo cubular npc $X$ es especial si y solo si $\pi_1X$ es el subgrupo de un gaar.
\end{corollary}

\begin{proof}
Una direcci\'on es el teorema~\ref{thm:specialvetti} m\'as el teorema~\ref{thm:local}, la otra es el teorema~\ref{thm:salvspe}. 
\end{proof}

Hay que detenerse a apreciar tantito la importancia del corolario anterior: nos esta diciendo, primero que nada, que empezando con una clase de complejos cubulares definidos a priori de una manera que se antoja b\'asicamente arbitraria, llegamos a una caracterizaci\'on algebraica en terminos de una clase de grupos relativamente bien entendida y que satisface propiedades muy deseables. Esta caracterizaci\'on, adem\'as, nos esta diciendo que `ser especial' es una propiedad que los grupos satisfacen \textbf{de manera abstracta}.

Esto motiva la siguiente definici\'on:

\begin{definition}
Un grupo $G$ es \emph{especial} si es isomorfo al grupo fundamental de un complejo cubular especial. 
\end{definition}

M\'as en general, decimos que $G$ es \emph{virtualmente especial} si tiene un subgrupo de \'indice finito isomorfo al grupo fundamental de un complejo cubular especial. Vale la pena mencionar que en el contexto hiperb\'olico ser virtualmente especial es independiente de la cubulaci\'on:

\begin{theorem}
Si $X$ y $Y$ son complejos cubulares npc compactos y sus grupos fundamentales son isomorfos e hiperb\'olicos, entonces
$X$ es virtualmente especial si y solo si $Y$ es virtualmente especial.
\end{theorem}

\subsection{Separabilidad}

En topolog\'ia de dimensiones bajas, el problema de `promover' o `mejorar' inmersiones de ciertas subvariedades (i.e., curvas o  superficies) a encajes es fundamental. La forma mas cl\'asica de este problema se concierne con levantar inmersiones a encajes en cubrientes de grado finito. Algebraicamente esto se relaciona con la siguiente propiedad:   

\begin{definition} Un grupo $G$ es \emph{residualmente finito} si para todo $g \in G - \{1\}$ hay un cociente finito de $G$ en el que la imagen de $g$ no es trivial. Esto es, $G$ es residualmente finito si para todo $g \in G - \{1\}$, existe $H<G$ tal que $[G:H]< \infty$ y $g\notin H$.
\end{definition}

M\'as en general:

\begin{definition}
Un subgrupo $H < G$ es \emph{separable} si para cada $g \in G-H$, existe $G' < G$ tal que $[G:G']<\infty$, $g\notin G'$ y $H< G'$.
\end{definition}

\begin{remark}
El subgrupo $\{1_G\}$ es separable si y solo si $G$ es residualmente finito.
\end{remark}

Marshall Hall demostr\'o en 1949 que los grupos libres satisfacen buenas propiedades de separabilidad. Recordemos que un subgrupo  $H \subset G$ es un \emph{retracto} si existe un endomorfismo $f$ de G tal que $f^2 = f$  y $f(G)=H$.

\begin{theorem}\label{thm:freesep}
Los grupos libres tienen las siguientes propiedades:

\begin{enumerate}
\item Son residualmente finitos,
\item todos sus subgrupos finitamente generados son retractos  de subgrupos de \'indice finito,
\item todos sus subgrupos finitamente generados son separables.
\end{enumerate}

\end{theorem}

Peter Scott~\cite{Scott78} extendi\'o este resultado a grupos de superficies:

\begin{theorem}\label{teo:scott}
Si $G=\pi_1 S$ para alguna superficie $S$, entonces todo subgrupo finitamente generado $H < G$ es separable.
\end{theorem}

Volviendo al contexto geom\'etrico, el teorema anterior implica que si $\alpha$ es un representante de una clase de homotop\'ia $[\alpha] \in \pi_1 S$, entonces existe un cubriente finito $\hat S \rightarrow S$ donde $\alpha$ se levanta a un encaje\footnote{Realmente tenemos que imponer una condici\'on extra en $\alpha$ para que esto se cumpla: es necesario que $\alpha$ tenga intersecciones m\'inimas en su clase de homotop\'ia, porque si $\alpha$ no tiene intersecciones m\'inimas entonces hay discos en $S$ formados por segmentos de $\alpha$ que necesariamente se levantan a la cubierta universal de $S$.}. En efecto, la separabilidad de $\langle [\alpha]\rangle$ implica que podemos encontrar un subgrupo de \'indice finito $G' < G$ que contiene a $\langle [\alpha]\rangle$, pero donde los elementos de $\pi_1S$ que corresponden a sublazos de $\alpha$ no sobreviven. 

Uno de las razones por las que los complejos especiales son \'utiles es porque sus grupos fundamentales tienen propiedad similares a las de los grupos libres. En particular en lo que concierne a la separabilidad de sus subgrupos:

\begin{theorem}\label{thm:doublesep}~\cite{WiseCBMS2012} Un complejo cubular npc $X$ es virtualmente especial si y solo si 
\begin{enumerate}
\item $\pi_1H$ es separable para cada hiperplano $H$ de $X$ y
\item la clase lateral doble $\pi_1H\pi_1J$ es separable para cada par de hiperplanos $H,J$ que se cruzan en $X$.
\end{enumerate} 
\end{theorem}

En el contexto hiperb\'olico, tenemos el siguiente resultado:

\begin{theorem}
Si $X$ es compacto y especial y $\pi_1X$ es hiperb\'olico, entonces todos los subgrupos cuasiconvexos de $\pi_1X$ son separables.
\end{theorem}

En la introducci\'on de este art\'iculo hicimos aluci\'on a la \emph{conjetura virtual de Haken}. Una $3$-variedad (digamos compacta, orientable e irreducible) es \emph{Haken} si contiene una superficie encajada, orientable e incompresible. No todas las $3$-variedades  con las propiedades antes mencionadas son Haken, pero la conjetura virtual de Haken plantea la posibilidad de que todas tengan cubrientes de grado finito que s\'i lo sean. Hasta hace algunos a\~nos, este era uno de los problemas abiertos m\'as importantes en topolog\'ia de dimensiones bajas. El lector podr\'a apreciar que los resultados mencionados en esta secci\'on (el teorema~\ref{thm:freesep} y el teorema~\ref{teo:scott}) son an\'alogos a esta situaci\'on en dimensiones $1$ y $2$.

La conjetura virtual de Haken fue resuelta por Ian Agol~\cite{AgolGrovesManning2012}, bas\'andose fuertemente en los resultados de Dani Wise y de sus colaboradores~\cite{BergeronWiseBoundary, HaglundWiseAmalgams, HsuWiseCubulatingMalnormal}: en particular, muchas de las nociones y resultados expuestos en este art\'iculo, y otros m\'as que est\'an  mucho m\'as all\'a de los objetivos de este texto, como la \emph{teor\'ia cubular de cancelaciones peque\~nas} y el \emph{teorema del cociente malnormal especial}~\cite{WiseIsraelHierarchy}.

\subsubsection{¿Qu\'e tan com\'un es ser especial?}

A estas alturas, quiz\'a se estar\'an preguntando si no ser\'a que \emph{todos} los complejos cubulares npc son virtualmente especiales, o, como m\'inimo, si no ser\'a que todos los grupos cubulables son virtualmente especiales. En el caso de los grupos, si asumimos adem\'as hiperbolicidad, esto se sigue de un teorema tremendo de Agol~\cite{AgolGrovesManning2012}:

\begin{theorem}
Si $G$ es hiperb\'olico y cocompactamente cubulado, entonces $G$ es virtualmente especial.
\end{theorem}

De esto se sigue, por ejemplo, que:

\begin{enumerate}
\item  Los grupos que satisfacen las condicion $C'(\frac{1}{6})$ de cancelaciones peque\~nas son virtualmente especiales~\cite{WiseSmallCanCube04}.
\item  Los grupos de Coxeter hiperb\'olicos son virtualmente especiales, puesto que tienen subgrupos de \'indice finito que est\'an cocompactamente cubulados (de hecho esto es cierto para \textbf{todos} los grupos de Coxeter, hiperb\'olicos o no~\cite{HaglundWiseCoxeter}).
\item  Los grupos fundamentales de $3$-variedades hiperb\'olicas cerradas son virtualmente especiales (hiperbolicidad en el sentido geom\'etrico implica hiperbolicidad para los grupos fundamentales por el teorema de Milnor-Svarc).
\end{enumerate}

Fuera del contexto hiperb\'olico esto ya no se cumple: hay un ejemplo de Wise~\cite{WiseCSC} de un complejo cubular npc que no es especial y que no tiene cubrientes de grado finito, y otros ejemplos con propiedades similares de Burger-Moses~\cite{BurgerMozes2000} y de Wise~\cite{WiseCSC}. 

\begin{figure}[h!]
\centerline{\includegraphics[scale=0.52]{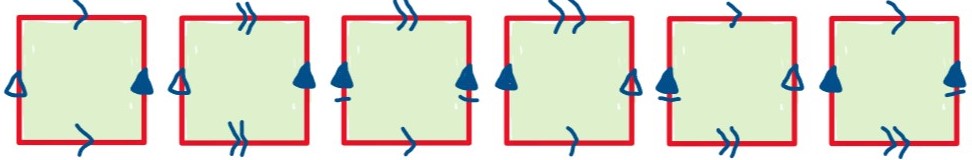}}
\caption{Un complejo cubular npc cuyo grupo fundamental no tiene subgrupos de \'indice finito se construye a partir de este complejo. Este complejo no es especial pues los hiperplanos se autoosculan.}
\label{fig:nonresf}
\end{figure}

\newpage

\section{Ejercicios}

\begin{enumerate}

\item \label{ejer:1} Sea $S$ una superficie cerrada con caracter\'istica de Euler $\leq 0$. Vimos en el art\'iculo como dar a $S$ una estructura de complejo cubular no positivamente curvado. Hay muchas maneras de hacer esto. Encuentra otras maneras de cuadricular a $S$ para obtener un complejo cubular no positivamente curvado. 

\item Demostrar que la subdivisi\'on baric\'entrica de un complejo cubular npc es de nuevo un complejo cubular npc.

\item Demostrar que si $A,B$ son complejos cubulares no positivamente curvados, entonces $A\times B$ tambi\'en lo es. (Hint: Demostrar que $link((a,b))$ es la \emph{junta} de  $link(a)$ y $link(b)$. La junta de dos espacios topol\'ogicos $A$ y $B$ es el espacio que se obtiene al tomar la union disjunta de los dos espacios y a\~nadir un segmento conectando cada par de puntos $(a,b)$ donde  $a \in A$ y  $b \in B$.)

\item Demostrar que un complejo cubular simplemente conexo cuyas aureolas son gr\'aficas bipartitas completas es el producto de dos \'arboles.

\item Si $\Gamma$ es a una gr\'afica $n$-partita completa, ¿qu\'e puedes decir sobre la estructura de $G(\Gamma)$?

\item Identificar los gaars asociados a las siguientes gr\'aficas y describir los complejos de Salvetti correspondientes. Para cada uno de estos complejos de Salvetti, describir las aureolas de sus v\'ertices.
\begin{figure}[h!]
\centerline{\includegraphics[scale=0.45]{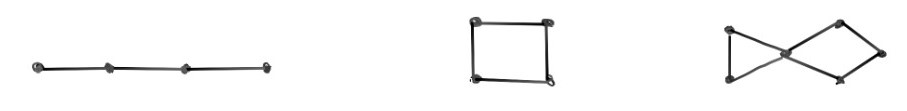}}
\label{fig:exsraags}
\end{figure}

\item En el Ejercicio~\eqref{ejer:1}, ¿puedes cuadricular a $S$ de tal forma que haya un solo hiperplano? En t\'erminos del g\'enero de $S$, ¿cu\'antos cuadrados son necesarios?

\item Demostrar que si $X$ es un complejo de Salvetti, entonces cada hiperplano de $X$ tambi\'en es un complejo de Salvetti. 

\item Construir un complejo cubular no positivamente curvado $X$  con $dim(X)  \geq 2$ y con un solo hiperplano. Si $dim(X) =2$, ¿cu\'al es el mínimo n\'umero de cuadrados necesarios? Si $dim(X)= 3$, ¿cu\'al es el mínimo n\'umero de cubos?

\item Construir un complejo cubular no positivamente curvado cuyos hiperplanos sean todos isomorfos a: (a) un c\'irculo, (b) un bouquet de dos c\'irculos, (c) la cu\~na de un toro con un c\'irculo.

\item Demostrar que la inclusi\'on de un cubo $c \rightarrow X$ en un complejo cubular npc $X$ es una isometr\'ia local.

\item Demostrar que, despu\'es de subdivir $X$, el mapeo $H \hookrightarrow X$ de un hiperplano $H$ a $X$ es una isometr\'ia local (Hint: seguir la demostrac\'ion del lema~\ref{lem:isosop}).

\item Demostrar que la intersecci\'on de subcomplejos convexos es convexa.

\item Sea $S$ una superficie cerrada y con $\chi(S)< 0$. Demostrar que existe un subgrupo libre $F_2 < \pi_1(S)$ usando resultados vistos en este art\'iculo.

\item Usando los resultados vistos en este art\'iculo, demostrar que si $G(\Gamma)$ es un gaar y $\Gamma'$ es una subgr\'afica plena de $\Gamma$, entonces $G(\Gamma') < G(\Gamma)$.

\item Dibujar la gr\'afica de Cayley de $\langle a | a^n \rangle$ donde $n \in \naturals$ con respecto al generador dado, la gr\'afica de Cayley de $\integers^2$ con respecto a $\{(1,0),(0,1)\}$, y las gr\'aficas de Cayley de las presentaciones dadas en el ejemplo~\ref{ex:surface}.

\item Demostrar que existe una $\delta \geq 0$ para la cual todos los tri\'angulos geod\'esicos en el espacio hiperb\'olico $\mathbb{H}^n$ son $\delta$-finos en el sentido de la definici\'on~\ref{def:finetriangle}. Demostrar que este no es el case en el espacio euclideano $\mathbb{E}^n$. 

\item Para cada $n \in \naturals$ y $n'>1$, dar ejemplos de presentaciones que satisfagan la condici\'on $C(n)$ de cancelaciones peque\~nas pero no la condici\'on $C'(\frac{1}{n'})$ de cancelaciones peque\~nas.

\item Determinar el complejo cubular dual a los siguientes espacios de paredes, y su cociente bajo la acci\'on por traslaciones de $\integers$:

\begin{figure}[h!]
\centerline{\includegraphics[scale=0.5]{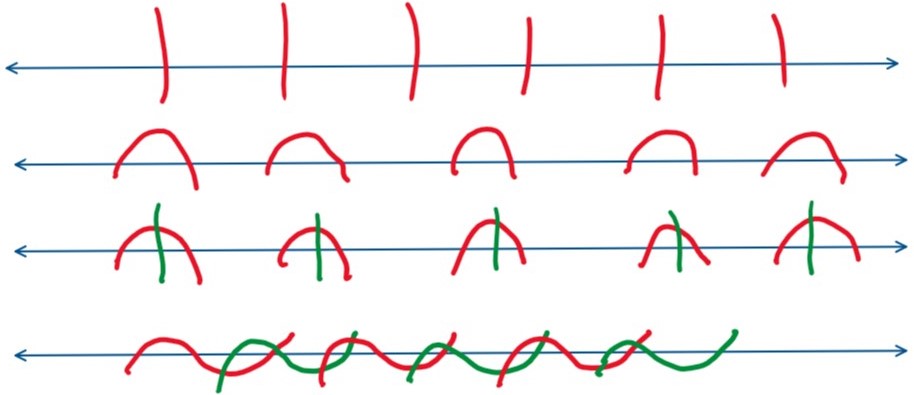}}
\label{fig:paredesr}
\end{figure}

\item Determinar el complejo cubular dual al espacio de paredes dado por una colecci\'on de $n$-l\'ineas donde cada par de l\'ineas se cruza y por sus trasladados enteros en $\mathbb{E}^2$ (ver la figura~\ref{fig:nlines}).

\begin{figure}[h!]
\centerline{\includegraphics[scale=0.55]{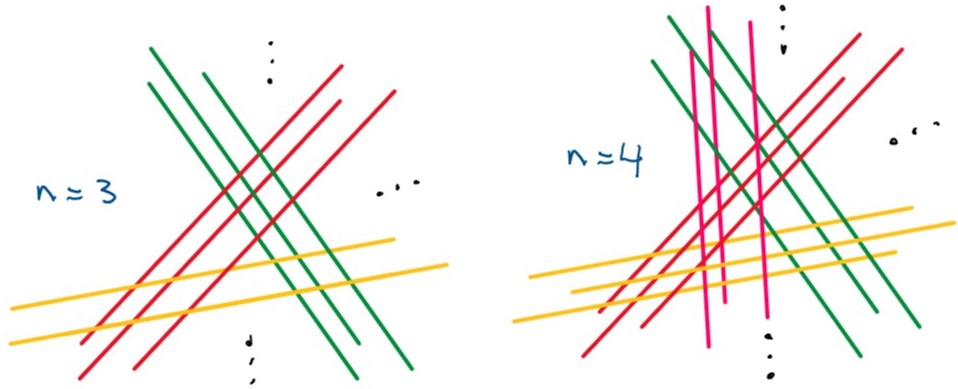}}
\caption{Colecciones de 3 y 4 l\'ineas donde cada par de l\'ineas se cruza, y algunos de sus trasladados.}
\label{fig:nlines}
\end{figure}

\item Sea $(X, \mathcal{W})$ un espacio de paredes y $\mathcal{C}$ su complejo cubular dual. Demostar que un cubo maximal de dimensi\'on $n$ en $\mathcal{C}$ corresponde a una colecci\'on maximal de $n$ paredes que se cruzan todas por pares.  

\item Para cualquier colecci\'on de geod\'esicas cerradas en una superficie $S$, sus levantamientos a la cubierta universal $\widetilde S$ (el plano hiperb\'olico, si $S$ es cerrada y no es un toro o una botella de Klein) inducen una estructura de espacio de paredes en $\widetilde S$. Como ya mencionamos antes $\pi_1 S$ act\'ua en el complejo dual $C$. ¿Cu\'ando es la acci\'on propia? ¿Cu\'ando es cocompacta? ¿Cu\'ando es el cociente de $C$ bajo la acci\'on homeomorfo a $S$? 

\item Completar la demostraci\'on del lema~\ref{lem:projects}.

\item Demostrar que las superficies cerradas con caracter\'istica de Euler $\leq 0$ son homeomorfas a complejos virtualmente especiales. ¿Cu\'al es el g\'enero m\'inimo necesario para que una superficie sea homeomorfa a un complejo especial? 

\item Encontrar una cubierta especial del siguiente complejo cubular npc.

\begin{figure}[h!]
\centerline{\includegraphics[scale=0.52]{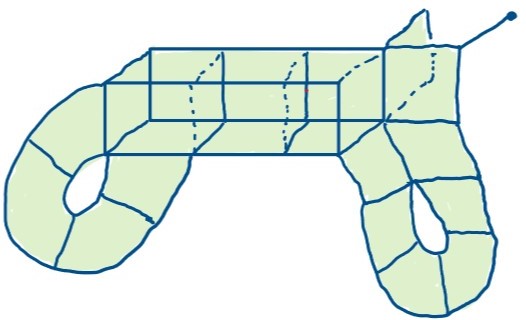}}
\label{fig:espex}
\end{figure}

\item Determinar el n\'umero mínimo de cuadrados necesarios para construir un complejo cubular especial de dimensión 2. Determinar el n\'umero mínimo de cubos necesarios para construir un complejo cubular especial de dimensión 3. ¿C\'omo cambian estos números si solo pedimos que el complejo sea virtualmente especial? 

\item ¿Cu\'ando puede un gaar ser un grupo hiperb\'olico?

\item Reformular la condici\'on de ser virtualmente especial para grupos en t\'erminos de la acci\'on en un complejo cubular CAT(0).

\end{enumerate}

\bibliographystyle{alpha}
\bibliography{wise.bib}

\def\cprime{$'$} \def\polhk#1{\setbox0=\hbox{#1}{\ooalign{\hidewidth
  \lower1.5ex\hbox{`}\hidewidth\crcr\unhbox0}}} \def\cprime{$'$}
  \def\cprime{$'$} \def\polhk#1{\setbox0=\hbox{#1}{\ooalign{\hidewidth
  \lower1.5ex\hbox{`}\hidewidth\crcr\unhbox0}}}
\begin{thebibliography}{BSV14}

\bibitem[Ago08]{Agol08}
Ian Agol.
\newblock Criteria for virtual fibering.
\newblock {\em J. Topol.}, 1(2):269--284, 2008.

\bibitem[Ago13]{AgolGrovesManning2012}
Ian Agol.
\newblock The virtual {H}aken conjecture.
\newblock {\em Doc. Math.}, 18:1045--1087, 2013.
\newblock With an appendix by Agol, Daniel Groves, and Jason Manning.

\bibitem[Are23]{arenas2023asphericity}
Macarena Arenas.
\newblock Asphericity of cubical presentations: the 2-dimensional case, 2023.

\bibitem[BH99]{BridsonHaefliger}
Martin~R. Bridson and Andr{\'e} Haefliger.
\newblock {\em Metric spaces of non-positive curvature}.
\newblock Springer-Verlag, Berlin, 1999.

\bibitem[BM00]{BurgerMozes2000}
Marc Burger and Shahar Mozes.
\newblock Lattices in product of trees.
\newblock {\em Inst. Hautes \'Etudes Sci. Publ. Math.}, (92):151--194 (2001),
  2000.

\bibitem[BSV14]{GGTbook14}
Mladen Bestvina, Michah Sageev, and Karen Vogtmann, editors.
\newblock {\em Geometric group theory}, volume~21 of {\em IAS/Park City
  Mathematics Series}.
\newblock American Mathematical Society, Providence, RI; Institute for Advanced
  Study (IAS), Princeton, NJ, 2014.
\newblock Including lecture notes from the Graduate Summer School held at the
  Park City Mathematics Institute (PCMI), Park City, UT, July 1--21, 2012.

\bibitem[BW12]{BergeronWiseBoundary}
Nicolas Bergeron and Daniel~T. Wise.
\newblock A boundary criterion for cubulation.
\newblock {\em Amer. J. Math.}, 134(3):843--859, 2012.

\bibitem[CN05]{ChatterjiNiblo04}
Indira Chatterji and Graham Niblo.
\newblock From wall spaces to {$\rm CAT(0)$} cube complexes.
\newblock {\em Internat. J. Algebra Comput.}, 15(5-6):875--885, 2005.

\bibitem[Deh87]{MaxDehnBook}
Max Dehn.
\newblock {\em Papers on group theory and topology}.
\newblock Springer-Verlag, New York-Berlin, 1987.
\newblock Translated from the German and with introductions and an appendix by
  John Stillwell, With an appendix by Otto Schreier.

\bibitem[Dro83]{DromsPhD83}
Carl Droms.
\newblock {\em Graph Groups}.
\newblock PhD thesis, Syracuse University, 1983.

\bibitem[Gro87]{Gromov87}
M.~Gromov.
\newblock Hyperbolic groups.
\newblock In {\em Essays in group theory}, volume~8 of {\em Math. Sci. Res.
  Inst. Publ.}, pages 75--263. Springer, New York, 1987.

\bibitem[Hag08]{HaglundGraphProduct}
Fr{\'e}d{\'e}ric Haglund.
\newblock Finite index subgroups of graph products.
\newblock {\em Geom. Dedicata}, 135:167--209, 2008.

\bibitem[HP98]{HaglundPaulin98}
Fr{\'e}d{\'e}ric Haglund and Fr{\'e}d{\'e}ric Paulin.
\newblock Simplicit\'e de groupes d'automorphismes d'espaces \`a courbure
  n\'egative.
\newblock In {\em The Epstein birthday schrift}, pages 181--248 (electronic).
  Geom. Topol., Coventry, 1998.

\bibitem[HW08]{HaglundWiseSpecial}
Fr{\'e}d{\'e}ric Haglund and Daniel~T. Wise.
\newblock Special cube complexes.
\newblock {\em Geom. Funct. Anal.}, 17(5):1 551--1620, 2008.

\bibitem[HW10]{HaglundWiseCoxeter}
Fr{\'e}d{\'e}ric Haglund and Daniel~T. Wise.
\newblock Coxeter groups are virtually special.
\newblock {\em Adv. Math.}, 224(5):1890--1903, 2010.

\bibitem[HW12]{HaglundWiseAmalgams}
Fr{\'e}d{\'e}ric Haglund and Daniel~T. Wise.
\newblock A combination theorem for special cube complexes.
\newblock {\em Ann. of Math. (2)}, 176(3):1427--1482, 2012.

\bibitem[HW15]{HsuWiseCubulatingMalnormal}
Tim Hsu and Daniel~T. Wise.
\newblock Cubulating malnormal amalgams.
\newblock {\em Invent. Math.}, 199:293--331, 2015.

\bibitem[NR98]{NibloRoller98}
Graham~A. Niblo and Martin~A. Roller.
\newblock Groups acting on cubes and {K}azhdan's property ({T}).
\newblock {\em Proc. Amer. Math. Soc.}, 126(3):693--699, 1998.

\bibitem[Sag95]{Sageev95}
Michah Sageev.
\newblock Ends of group pairs and non-positively curved cube complexes.
\newblock {\em Proc. London Math. Soc. (3)}, 71(3):585--617, 1995.

\bibitem[Sag97]{Sageev97}
Michah Sageev.
\newblock Codimension-$1$ subgroups and splittings of groups.
\newblock {\em J. Algebra}, 189(2):377--389, 1997.

\bibitem[Sco78]{Scott78}
Peter Scott.
\newblock Subgroups of surface groups are almost geometric.
\newblock {\em J. London Math. Soc. (2)}, 17(3):555--565, 1978.

\bibitem[Wis]{WiseIsraelHierarchy}
Daniel~T. Wise.
\newblock {\em The structure of groups with a quasiconvex hierarchy}.
\newblock Annals of Mathematics Studies, to appear.

\bibitem[Wis04]{WiseSmallCanCube04}
Daniel~T. Wise.
\newblock Cubulating small cancellation groups.
\newblock {\em GAFA, Geom. Funct. Anal.}, 14(1):150--214, 2004.

\bibitem[Wis07]{WiseCSC}
Daniel~T. Wise.
\newblock Complete square complexes.
\newblock {\em Comment. Math. Helv.}, 82(4):683--724, 2007.

\bibitem[Wis12]{WiseCBMS2012}
Daniel~T. Wise.
\newblock {\em From riches to raags: 3-manifolds, right-angled {A}rtin groups,
  and cubical geometry}, volume 117 of {\em CBMS Regional Conference Series in
  Mathematics}.
\newblock Published for the Conference Board of the Mathematical Sciences,
  Washington, DC, 2012.

\end{thebibliography}

\end{document}